\documentclass[1p]{article}

\usepackage{amsmath,amssymb}
\usepackage{amsthm}
\usepackage{enumerate}
\usepackage{multicol}
\usepackage{url}

\usepackage{color}

\usepackage{listings}
\usepackage{caption}

\newcounter{nalg} 

\renewcommand{\thenalg}{\arabic{nalg}} 
\DeclareCaptionLabelFormat{algocaption}{Algorithm \thenalg} 

\lstnewenvironment{algorithm}[1][] 
{
	\refstepcounter{nalg} 
	\captionsetup{labelformat=algocaption,labelsep=colon} 
	\lstset{ 
		mathescape=true,
		frame=tb,
		numbers=left,
		basicstyle={\small\sffamily},
		numberstyle=\tiny,
		showstringspaces=false,
		texcl=true,
		keywordstyle=\color{black}\bfseries\em,
		keywords={,input, output, put, remove, take, add, return, set, datatype, in, if, else, do, foreach, while, begin, end, procedure, subprocedure, append, } 
		numbers=left,
		breaklines=true,
		tabsize=3,
		xleftmargin=.04\textwidth,
		#1 
	}
}
{}

\usepackage{tikz}

\newtheorem{theorem}{Theorem}

\newtheorem{proposition}[theorem]{Proposition}
\newtheorem{corollary}[theorem]{Corollary}

\theoremstyle{definition}
\newtheorem{definition}[theorem]{Definition}

\newcommand{\BG}[1][U\sqcup V, E]{G=(#1)}
\newcommand{\arc}[1][uv]{\overrightarrow{#1}}
\newcommand{\Nodes}{U\sqcup V}

		\title{The equivalence between two classic algorithms for the assignment problem}
		\author{Carlos~A.~Alfaro
		\thanks{Banco de M\'exico, Mexico City, Mexico (carlos.alfaro@banxico.org.mx)}
		\and Sergio~L.~Perez
		\thanks{Banco de M\'exico, Mexico City, Mexico (sergio.perez@banxico.org.mx)}
		\and Carlos~E.~Valencia 
		\thanks{Departamento de Matem\'aticas, CINVESTAV del IPN, Apartado postal 14-740, 07000 Mexico City, Mexico (cvalencia@math.cinvestav.edu.mx)}
		\and Marcos~C.~Vargas 
		\thanks{Banco de M\'exico, Mexico City, Mexico (marcos.vargas@banxico.org.mx)}
		}

\begin{document}
	
	\maketitle
	
		\begin{abstract}
			We give a detailed review of two algorithms that solve the minimization case of the assignment problem.
			The Bertsekas' auction algorithm and the Goldberg \& Kennedy algorithm.
			We will show that these algorithms are equivalent in the sense that both perform equivalent steps in the same order.
			We also present experimental results comparing the performance of three algorithms for the assignment problem.
			They show the auction algorithm 	performs and scales better in practice than algorithms that are harder to implement but have better theoretical time complexity.
		\end{abstract}

	\section{Introduction}
	The assignment problem\index{assignment problem} can be stated as follows.
	We have a bipartite graph $G=(U\sqcup V, E)$, with $|U|=|V|$,
	and an \textit{integer weight function}\index{integer weight function}
	over the edges of $G$ given by $w:E\rightarrow \mathbb{Z}$.
	The objective is to find a perfect matching of $G$ of minimum weight.
	We call the pair $\{G, w\}$ an \textit{instance of the assignment problem}\index{instance of the assignment problem},
	or equivalently an \textit{integer weighted bipartite graph}\index{integer weighted bipartite graph}.
	A perfect matching of minimum weight is called an \textit{optimum matching}\index{optimum matching}.
	If the graph has no perfect matchings, then the problem is \textit{infeasible}.
	
	The assignment problem is important from a theoretical point of view because it appears as a subproblem of a vast number of combinatorial optimization problems, and its solution allows the development of algorithms to solve other combinatorial optimization problems.
	
	The roots of the assignment problem can be traced back to the 1920's studies on matching problems.
	And more remarkably by the 1935 marriage theorem of Philip Hall.
	The assignment problem can be modeled as a linear program, with the property that its associated polyhedron has all the vertices integer valued.
	Therefore, this problem can be solved using general linear programming techniques.
	The problem with these techniques is that they do not perform well in practice.
	This has pushed the development of specialized algorithms that exploit the particular structure of the assignment problem.
	For instance, Kuhn proposed in \cite{hung_kuhn1} the first polynomial time algorithm for solving the assignment problem.
	Since then, the assignment problem has been deeply studied, see for instance \cite{hung_kuhn2,asym_auct1,asym_auct2,hybr_auct1,hung_kuhn1,hung_kuhn2,gold_ken_imp1}.
	For a more detailed account of this fascinating topic, we refer the reader to \cite{assignment_problems_book}.
	The \textit{$\epsilon$-scaling auction algorithm} \cite{auctionA_Bert1} and the \textit{Goldberg \& Kennedy algorithm} \cite{Gold_ken2} also solve the assignment problem and for long time they have been considered as different algorithms, for instance see Section 4.1.3 of \cite{assignment_problems_book}.
	However, in this paper it will be shown that they are equivalent in the sense that one can be transformed into the other by means of memory optimization.
	
	Through the paper, we will use the following notation: $m=|E|$, $n=|U|=|V|$,
	and $W=\max_{uv\in E} |w(uv)|$.
	The theoretical time complexity of these two algorithms is not the best currently known.
	The best current time complexity for the assignment problem is $O(\sqrt{n}\,m\log (nW))$ proposed by Gabow \& Tarjan \cite{gabo_tarj1},
	while the time complexity for the $\epsilon$-scaling auction algorithm is $O(n\,m\log (nW))$,
	as well as for the Goldberg \& Kennedy algorithm.
	In our experience, the auction algorithm performs a much better in practice than other algorithms with equal or better theoretical time complexity, including the Goldberg \& Kennedy algorithm.
	
	The $\epsilon$-scaling auction algorithm has been modified onto several variants to directly solve related problems such as \textit{shortest path} \cite{auctionA_Bert4}, \textit{min-cost flow} (see \cite{auctionA_Bert2} p. 17, \cite{auctionA_Bert1}), \textit{assignment}, \textit{transportation} \cite{auctionA_Bert1} and many others.
	An important advantage of the auction algorithm is that it can be implemented using parallel computation \cite{paral_auction1}.
	It is important to remark that the auction algorithm only works with balanced instances, that is, $|U|=|V|$.
	
	The $\epsilon$-scaling auction algorithm operates like a real auction, where a set of persons $U$, compete for a set of objects $V$.
	In this scenario, to each object is assigned a price which, in certain sense, represents how important is the object for the persons.
	The optimization process is done in a competitive bidding, where the prices of
	the objects are properly reduced in order to make the desired object of a person less desirable to the other persons.
	
	On the other side, the Goldberg \& Kennedy algorithm is based on network flow techniques.
	The algorithm transforms the problem into a
	minimum-cost flow problem, and aims to build an optimum flow on a couple of auxiliary digraphs.
	Once we get an optimum flow, the induced optimum matching is easily obtained.

	\section{The $\epsilon$-scaling Auction algorithm}\label{auctionminimization}
	
	In the $\epsilon$-scaling auction algorithm every object $v$ has a price $p(v)$.
	This set of prices can be seen as a price function over $V$ defined by
	$p:V\rightarrow \mathbb{R}$.
	For every edge $uv\in E$ we define the \textit{reduced cost} $w'(uv)$ of job $v$ for the person $u$ as $w(uv) - p(v)$.
	
	The objective in the auction algorithm is to find prices $p(v)$ for the jobs and a perfect matching $M$ such that every person is assigned to the job that has almost the minimum reduced cost that the person can get.
	This condition is expressed in the following definition.
	\begin{definition}
		Given $\epsilon>0$.
		A set of prices $p$ and a perfect matching $M$ are said to satisfy the
		\textit{$\epsilon$-Complementary Slackness condition},
		\index{$\epsilon$-Complementary Slackness Condition}
		or \textit{$\epsilon$-CS condition}\index{$\epsilon$-CS condition} for short,
		if they satisfy:
		$$w'(uv) \leq \min_{z\in N(u)} w'(uz) + \epsilon, \qquad \forall uv\in M.$$
		\label{eps_cs_cond_def}
	\end{definition}
	
	The following theorem shows the reason of why this condition is important.
	\begin{theorem}
		Let $M^*$ be a perfect matching of minimum weight and $\epsilon>0$.
		If a perfect matching $M$ satisfies the $\epsilon$-CS condition
		with a set of prices $p$, then $w(M^*)\leq w(M)\leq w(M^*)+n\epsilon$.
		\label{n_eps_bound}
	\end{theorem}
	\begin{proof}
		Part $w(M^*)\leq w(M)$ follows because $M^*$ is optimum. In the other hand, from the definition of $\epsilon$-CS condition follows that:
		\[w(M) = \sum_{uv\in M} w(uv) = \sum_{uv\in M} w(uv) - \sum_{v\in V} p(v) + \sum_{v\in V} p(v)
		= \sum_{uv\in M} (w(uv) - p(v)) +  \sum_{v\in V} p(v)\]
		\[\leq  \sum_{uv\in M^*} (w(uv) - p(v) + \epsilon) +  \sum_{v\in V} p(v)
		= \sum_{uv\in M^*} w(uv) + \sum_{uv\in M^*} \epsilon = w(M^*) + n\epsilon.\]
	\end{proof}
	Since we are dealing with integer-weighted bipartite graphs, the following corollary shows how to obtain an optimum matching via the $\epsilon$-CS condition.
	
	\begin{corollary}
		If a perfect matching $M$ and a set of prices $p$ satisfy the $\epsilon$-CS condition
		for $\epsilon < \frac{1}{n}$, then $M$ is of minimum weight.
		\label{opt_mat_eps}
	\end{corollary}
	\begin{proof}
		Let $w(M^*)$ be the optimum weight.
		Theorem \ref{n_eps_bound} implies that
		$w(M^*)\leq w(M) < w(M^*) + 1$.
		Since the weights are integral, then $w(M) = w(M^*)$.
	\end{proof}
	
	The pseudo-code given in Algorithm \ref{auction_01} shows how to obtain a perfect matching and a set of
	prices that satisfy the $\epsilon$-CS condition for a given $\epsilon>0$.
	If the given instance has perfect matchings, then the procedure always terminates
	with the correct output as we will see later.
	It is important to remark that if the instance has no perfect matchings then the
	procedure will fall into an infinite loop.
	This algorithm is called the \textit{auction algorithm}.
	
	\begin{algorithm}[caption={The auction algorithm}, label={auction_01}]
		input: an instance $\{G,w\}$, $\epsilon>0$ and prices $p$.
		output: a perfect matching $M$ and prices $p$ satisfying the $\epsilon$-CS condition.
		
		procedure auction($G,w,\epsilon,p$)
		$\qquad$$M=\phi$
		$\qquad$while $|M|<n$ do
		$\qquad$$\qquad$take an unassigned $u \in U$
		$\qquad$$\qquad$bid(u)
		$\qquad$end
		$\qquad$return($M,p$)
		end
		
		procedure bid($u$)
		$\qquad$Let $uv$ and $uz$ be the edges with the smallest and the second smallest reduced costs, respectively.
		$\qquad$if($v$ is assigned to some $u' \in U$)
		$\qquad$$\qquad$remove $u'v$ from $M$
		$\qquad$append $uv$ to $M$
		$\qquad$$\gamma=w'(uz)-w'(uv)$
		$\qquad$$p(v)=p(v)-\gamma - \epsilon$
		end
	\end{algorithm}
	
	Note that we allow the procedure to receive initial prices $p$.
	Such initial prices have no particular restrictions, they can be any real values.
	The algorithm will automatically adjust them after each iteration and will end up with the correct output as the following proposition shows,
	which is proven in \cite{auctionA_Bert1}.
	However, as we will see later, the initial prices have a big impact on the running time of the procedure.

	\begin{proposition}
		If the auction procedure, described in Algorithm \ref{auction_01}, is applied to a feasible instance of the assignment problem,
		then the procedure will terminate after a finite number of
		iterations and will return a matching and a set of prices that satisfy the $\epsilon$-CS condition,
		regardless of the initial prices.
		\label{auct_terminates_prop}
	\end{proposition}
	
	It turns out that if the initial prices $p$ are random and $\epsilon$ is very close to 0,
	then the resulting matching will be optimum or close to optimum, but the
	running time will be huge, and if $\epsilon$ is large then the
	running time will be small but the matching will be far from being optimum.
	However, if the initial prices have a structure close to the $\epsilon$-CS condition,
	then the running time is proven to be $O(nm)$ \cite{auctionA_Bert1}.
	The $\epsilon$-scaling auction algorithm exploits this behavior to efficiently
	find a perfect matching $M$ that satisfies the $\epsilon$-CS condition for our small $\epsilon < \frac{1}{n}$.
	The pseudo-code is given in Algorithm \ref{eps_auction_01}, and makes use of the \textit{auction} procedure.
	The scaling factor $\alpha>1$ is a custom parameter and remains constant during the execution.
	
	\begin{algorithm}[caption={min\_epsilon\_scaling\_auction}, label={eps_auction_01}]
		input: a weighted bipartite graph $\{G,w\}$.
		output: a minimum weight perfect matching.
		
		procedure: $\epsilon$_scaling_auction($G,w$)
		$\qquad$$\epsilon=W$
		$\qquad$$p(v)=0$, $\forall v\in V$
		$\qquad$while $\epsilon \geq 1/n$ do
		$\qquad$$\qquad$$\epsilon=\epsilon / \alpha$
		$\qquad$$\qquad$$(M,\, p) =$auction$(G, w, \epsilon,\, p)$
		$\qquad$end
		$\qquad$return $M$
		end
	\end{algorithm}
	
	The idea of the $\epsilon$-scaling auction algorithm is to iteratively find a sequence of pairs $\{M, p\}$
	that satisfy the $\epsilon$-CS condition,
	where the sequence of values for $\epsilon$ is $\{W/\alpha, W/\alpha^2, \\W/\alpha^3, \ldots, W/\alpha^k \}$,
	with $W/\alpha^k < 1/n$.
	Note that the resulting
	prices that satisfy the $\epsilon$-CS condition
	for $\epsilon = W/\alpha^i$ almost satisfy the $\epsilon$-CS condition for the next iteration
	with $\epsilon = W/\alpha^{(i+1)}$.
	
	As we mentioned before, each call to the auction procedure takes $O(nm)$ time.
	And the number of scaling phases is $O(\log(nW))$.
	This gives us a total of $O(nm\log(nW))$ running time for the $\epsilon$-scaling auction
	algorithm. The following corollary follows directly from Proposition \ref{auct_terminates_prop}
	and Corollary \ref{opt_mat_eps}.
	\begin{corollary}
		If the $\epsilon$-scaling auction procedure is applied to a feasible instance of the assignment problem,
		then the procedure will terminate after a finite number of steps and will return an optimum matching.
		\label{eps_auct_terminates_coro}
	\end{corollary}
	


	\section{Goldberg \& Kennedy algorithm}
	
	The Goldberg \& Kennedy algorithm \cite{Gold_ken2} applies network flow techniques to the assignment problem, which is a special
	case the minimum-cost flow problem.
	Their algorithm is base on the push-relabel technique \cite{gold_tarj1}.
	The complexity of the resulting algorithm is $O(n\,m\,\log(nW))$ and an overview is presented following.
	
	Given a weighted bipartite graph $\{\BG, w:E\rightarrow \mathbb{Z}\}$, it is transformed into an instance of the min-cost flow problem
	$\{\hat{G}=(U\sqcup V, \hat{E}), w, c, d\}$.
	First, $\hat{G}$ is considered as a directed graph with the same vertex set $U\sqcup V$ and the arc set $\hat{E}$ equal to the edges of $E$ but oriented from $U$ to $V$.
	We denote the arc that goes from $u$ to $v$ by $\arc$.
	The {\it capacity function} is given by $c(\arc)=1$, $\forall \arc \in \hat{E}$.
	And the {\it supply function} is given by $d(u)=1$ $\forall u\in U$, and $d(v)=-1$ $\forall v\in V$.
	
	A \textit{pseudoflow} is a function $f:\hat{E} \rightarrow \mathbb{Z}$ such that
	$f(\arc) \leq c(\arc)$ for every arc $\arc\in \hat{E}$.
	We define the \textit{excess flow} of a vertex $x\in \Nodes$ by:
	\begin{equation}
		e_f(x)=d(x) + \sum_{\arc[yx]\in \hat{E}} \; f(\arc[yx]) - \sum_{\arc[xy]\in E} \; f(\arc[xy]).
		\label{excess_flow_0}
	\end{equation}
	A node $v$ satisfying $e_f(v)>0$ is called \textit{active}.
	A \textit{flow} is a pseudoflow with no active nodes. The weight of a pseudoflow is:
	\begin{equation}
		w(f)= \sum_{\arc\in \hat{E}} \; w(\arc)\cdot f(\arc).
	\end{equation}
	The objective is to find a flow of minimum weight.
	
	Given a pseudoflow, the \textit{residual capacity} of an arc $\arc\in\hat{E}$ is $c_f(\arc)=c(\arc) - f(\arc)$.
	The \textit{residual graph} induced by the flow is $G_f=(\Nodes, E_f)$, where the set of \textit{residual arcs} $E_f$
	contains the arcs that satisfy $c_f(\arc)>0$ and the reversed arcs $\arc[vu]$ such that $c_f(\arc)=0$.
	The weight function $w_f$ defined over the residual arcs is $w_f(\arc)=w(\arc)$ for forward arcs,
	and $w_f(\arc[vu]) = -w(\arc)$ for reversed arcs.
	
	Let $p:\Nodes \rightarrow \mathbb{R}$ be a function that assigns prices to the vertices.
	The \textit{reduced cost} of an arc $\arc[xy]\in E_f$ is given by
	$w_p(\arc[xy]) =	w_f(\arc[xy]) + p(x) - p(y)$.
	The \textit{partial reduced cost} of an arc $\arc\in \hat{E}$ is
	$w'_p(\arc)=w(\arc)-p(v)$.
	
	Given $\epsilon>0$, a pseudoflow $f$ is said to be \textit{$\epsilon$-optimal}, with respect to the price
	function $p$, if every arc of $E_f$ satisfies the following:
	\begin{subequations}\label{for_rev_opt_cond_00}
		\begin{align}
			\mbox{For a reversed arc $\arc[vu]$:} &	\qquad\qquad\qquad w_p(\arc[vu]) \geq -\epsilon, \qquad\qquad\qquad\qquad\qquad	\label{for_rev_opt_cond_01}\\
			\mbox{For a forward arc $\arc$:} \; &	\qquad\qquad\qquad w_p(\arc) \geq 0. \qquad\qquad\qquad\qquad\qquad	\label{for_rev_opt_cond_02}
		\end{align}
	\end{subequations}
	
	\begin{theorem}
		Let $\epsilon<1/n$. If $f$ is $\epsilon$-optimal with a price function $p$, then it is of minimum weight.
		\label{optimality_flow_theorem_0}
	\end{theorem}
	The proof of this theorem is given in \cite{gold_proof1} and \cite{Gold_ken2}.
	Note that given a pseudoflow, a matching is induced by the arcs that carry one unit of flow.
	The full pseudocode of the algorithm is shown in Algorithm~\ref{Goldberg_Kennedy_01}.
	The real number $\alpha>1$ is a custom parameter and remains constant during the execution of the algorithm.
	During the algorithm, we assume that we construct and keep track of
	the directed graphs $\hat{G}$ and $G_f$.
	To keep track of the matching induced by the resulting flow, we use the variable $M$,
	where $M(v)=u$ if and only if $f(\arc)=1$.
	
	\begin{algorithm}[caption={Goldberg \& Kennedy algorithm}, label={Goldberg_Kennedy_01}]
		Input: a weighted bipartite graph $\{ G,\, w \}$.
		Output: a matching of minimum weight.
		
		procedure Goldberg_Kennedy($G, w$)
		$\qquad$$\epsilon=W$
		$\qquad$$p(v)=0$, $\forall v \in V$
		$\qquad$while $\epsilon \geq 1/n$ do
		$\qquad$$\qquad$$\epsilon=\epsilon / \alpha$
		$\qquad$$\qquad$$(f,\, p) =$refine$(G,\, w,\, \epsilon,\, p)$
		$\qquad$end
		$\qquad$return $M$ (Matching induced by $f$)
		end
		
		procedure refine($G,\, w,\, \epsilon,\, p$)
		$\qquad$$f(\arc)=0$, $\forall \arc\in \hat{E}$
		$\qquad$$p(u)=-\min_{\arc[uz]\in \hat{E}}\; w'_p(\arc[uz])$, $\forall u\in U$
		$\qquad$while $f$ is not a flow do
		$\qquad$$\qquad$take an active $u \in U$
		$\qquad$$\qquad$double_push(u)
		$\qquad$end
		$\qquad$return($f,p$)
		end
		
		procedure double_push($u$)
		$\qquad$let $\arc$ and $\arc[uz]$ be the arcs with the smallest and
		$\qquad$ second smallest partial reduced costs, respectively
		$\qquad$$p(u)=-w'_p(\arc[uz])$
		$\qquad$send one unit of flow from $u$ to $v$
		$\qquad$if($e_f(v)>0$)
		$\qquad$$\qquad$send one unit of flow from $v$ to its match $M(v)$
		$\qquad$$M(v)=u$
		$\qquad$$p(v)=p(u)+w(\arc) - \epsilon$
		end
	\end{algorithm}

	The following theorem, which proof can be found in \cite{Gold_ken2} and \cite{gold_proof1}, states the correctness of the algorithm.
	
	\begin{theorem}
		If $G$ is feasible and balanced, then the \textit{Goldberg\_Kennedy} procedure will finish and will return a matching of minimum weight.
	\end{theorem}

	\section{Equivalence of the $\epsilon$-scaling auction and the Goldberg \& Kennedy algorithms}

	It is worth mentioning that the pseudocode of the previous algorithms were presented a bit different from the original versions,
	while maintaining the same flow of operations.
	The objective in doing so is to make easier, to the reader, to observe the similitude between the algorithms.
	
	The central idea to see the equivalence is to make changes to the G\&K algorithm that turn it into the $\epsilon$-scaling auction algorithm.
	The changes are focused in getting rid of the auxiliary directed graphs used in the G\&K algorithm, which turn out to be redundant,
	inducing more expensive computations and an unnecessary increase of the memory space, as well as hard to follow.
	
	We are going to go through the Goldberg \& Kennedy algorithm as described in Algorithm~\ref{Goldberg_Kennedy_01}. 
	First note that \textit{reduced cost} from the auction algorithm and
	\textit{partial reduced cost} from Goldberg \& Kennedy are the same thing.
	Given an underlying edge $uv$, both are defined as $w(uv) - p(v)$.
	
	Let us go through the \textit{refine} procedure line-by-line to show the equivalence with the \textit{auction} procedure described in Algorithm~\ref{auction_01}.
	At first, the flow vanishes, which is the same as setting the induced matching empty.
	The second line only initializes the prices of the set $U$ which, as we will see later, are redundant
	and a waste of computation, we can safely remove this initialization.
	The {\sf while} condition is the same as testing whether the induced matching has less than $n$ edges.
	And finally, taking an active vertex is the same as taking a vertex which is not assigned under
	the induced matching.
	Therefore, we can rewrite the refine procedure as in Algorithm \ref{refine_01}.
	
	\begin{algorithm}[caption={Modified refine procedure}, label={refine_01}]
		procedure refine($G,\, w,\, \epsilon,\, p$)
		$\qquad$$M(v)=\phi$, $\forall v\in V$;
		$\qquad$while $|M| < n$ do
		$\qquad$$\qquad$take an unassigned $u \in U$;
		$\qquad$$\qquad$double_push(u);
		$\qquad$end
		$\qquad$return($M,p$);
		end
	\end{algorithm}
	
	Continuing with the \textit{double\_push} procedure of Algorithm~\ref{Goldberg_Kennedy_01}, the first instruction is clearly the same as the first instruction of the \textit{bid} procedure of Algorithm~\ref{auction_01}.
	For the second instruction, note that the previous value of the price of $u$ is completely ignored, because it is
	replaced by the minimum partial reduced cost given by its neighbours.
	This observation proves that the initialization discussed in the previous paragraph, inside the original refine procedure, is indeed redundant.
	Back to the \textit{double\_push} procedure, note that the only place where the price of $u$ is used is in the last instruction,
	therefore we can directly replace the computed price in this last instruction and forget about the second instruction.
	There are no more places in the pseudocode where the prices of $U$ show up, this shows that the prices for the vertices of $U$ are not needed.
	Once we substitute the computed price of the second line into the last line, the last instruction is equivalent to
	$p(v)=p(v) - \gamma - \epsilon$, where $\gamma=w_p'(uz)-w_p'(uv)$, as the following chain shows:	
	$$p(v) = -w_p'(uz) + w(uv) - \epsilon = p(v) - w_p'(uz) + w(uv) - p(v) - \epsilon = p(v) - (w_p'(uz) - w_p'(uv)) - \epsilon$$
	
	The instruction where we send one unit of flow from $u$ to $v$ and the instruction $M(v)=u$ in the double\_push procedure are equivalent,
	but since we want to get rid of the flow, we reject the first one and just conserve the second one.
	Finally, the condition in the \textit{if} instruction is equivalent to test if the vertex $v$ is already assigned under the induced matching,
	and sending one unit of flow from $v$ to $M(v)$ is equivalent to remove the edge $M(v)v$ from the matching.
	Then the double\_push procedure can be rewritten in the equivalent form shown in Algorithm \ref{double_push_02}.
	
	\begin{algorithm}[caption={Modified double\_push procedure}, label={double_push_02}]
		procedure double_push($u$)
		$\qquad$let $uv$ and $uz$ be the arcs with the smallest and
		$\qquad$ second smallest partial reduced costs, respectively;
		$\qquad$if($v$ is assigned to some $u'\in U$)
		$\qquad$$\qquad$$M(v)=\phi$;
		$\qquad$$M(v)=u$;
		$\qquad$$\gamma=w_p'(uz)-w_p'(uv)$;
		$\qquad$$p(v)=p(v) - \gamma - \epsilon$;
		end
	\end{algorithm}
	
	As we can see, we have removed the need of keeping the matching in two equivalent structures, the flow and the matching.
	Therefore we can also replace the flow by the matching in the Goldberg\_Kennedy procedure.
	If we compare side by side this equivalent version of the Goldberg \& Kennedy algorithm with the $\epsilon$-scaling auction algorithm,
	we will notice they are the same.
	
	One final thing to do is to make sure that the optimality condition of the Goldberg \& Kennedy algorithm is equivalent to the
	optimality condition of the $\epsilon$-scaling auction algorithm.
	Since the optimality condition makes use of the prices
	of the vertices of $U$, which were removed in the new versions,
	we will focus on the original pseudocode of the G\&K algorithm.
	
	If we keep track of the prices of the vertices of $U$ in the original pseudocode, it is not hard to find that at the
	end of every iteration in the loop, the price of any vertex $u\in U$ is equal to $p(u)=-min_{uz\in \hat{E}} w_p'(uz)$.
	This is clear in the initialization, but at the beginning of the double\_push procedure the price $p(u)$ is modified such
	that it is assigned the negative second minimum partial reduced cost for $u$, but at the end of the procedure the price $p(v)$ is
	also changed in such a way that the second minimum is now the minimum partial reduced cost.
	Therefore, the original optimality condition (\ref{for_rev_opt_cond_00}) is equivalent to the following:
	\begin{enumerate}
		\item For a reversed arc $\arc[vu]$ (i.e. for every $uv$ in the matching):
		\begin{equation}
			\begin{array}{rclc}
				w_p(\arc[vu]) & \geq &  - \epsilon &  \Longleftrightarrow \\
				w_f(\arc[vu]) + p(v) - p(u) & \geq & - \epsilon &  \Longleftrightarrow  \\
				-(w(\arc) - p(v))  & \geq & p(u) - \epsilon &  \Longleftrightarrow \\
				w(\arc) - p(v)  & \leq & -p(u) + \epsilon &  \Longleftrightarrow \\
				w_p'(uv) & \leq & min_{uz\in \hat{E}} w_p'(uz) + \epsilon, &
			\end{array}
		\end{equation}
		\item The second condition is redundant as we can see following, proceeding similar to the previous case.
		For a forward arc $\arc$ (i.e. for every other edge $uv$):
		$$w_p(\arc) \geq 0 \qquad\Longleftrightarrow\qquad w_p'(uv) \geq min_{uz\in \hat{E}} w_p'(uz)$$
	\end{enumerate}
	
	Which is the same as the optimality condition of the $\epsilon$-scaling auction algorithm.
	
	\section{Performance analysis}
	
	This section is dedicated to present experimental results about the performance of three different algorithms:
	the $\epsilon$-\emph{scaling auction algorithm} with theoretical time complexity of $O(nm\log(nW))$, the \emph{Hungarian algorithm} \cite{hung_kuhn1} with 
	time complexity of $O(mn+n^2\log n)$ and the \emph{FlowAssign algorithm} \cite{lyle_tarjan_1} with time complexity of $O(m\sqrt{s}\log(sW))$.
	Where $\{G=(U\sqcup V, E), w\}$ is a weighted bipartite graph not necessarily balanced, with $n=|U|\geq |V|=s$.
	The objective is to show that a very easy-to-understand and easy-to-implement algorithm like the auction algorithm
	performs and scales a lot better in practice than algorithms that are harder to implement and, as in the case of the FlowAssign algorithm, 
	have better theoretical time complexity.
	The reason we do not include the Goldberg \& Kennedy algorithm in this analysis is that we have just shown that this algorithm
	performs redundant computations and memory wastes, therefore it is expected a worst performance respect to
	the auction algorithm. Furthermore, we have also proven that if we optimize the implementation of the G\&K algorithm
	we will end up implementing the auction algorithm.
	
	To compare the performance of the algorithms we will construct different types of random instances with different structures.
	This is motivated by the fact that each algorithm may perform better on some graph structures than others and we want to explore
	how the algorithms behave to different graph configurations.
	
	An instance of the assignment problem has two components, the bipartite graph and the edge weights.
	We will define constructions for each component separately since we can pick one of each to form an instance.
	Given a bipartite graph $G$, we define its \textit{density} as $\displaystyle \rho(G) = \frac{|E|}{n\cdot s} \in [0,1]$.
	Note that in a complete bipartite graph $K_{n,s}$ we have $\rho(K_{n,s}) = 1$.
	
	\subsection{Models for generating random instances}
	
	The first model for generating random bipartite graphs is the \textit{Erd\"os-Renyi} model.
	Here, every possible edge of the bipartite graph
	has probability $d$ to stay in the graph, under a Bernoulli distribution.
	The pseudocode to generate this type of random bipartite graphs is given in the following procedure, which takes as argument: the number of vertices in
	$U$, the number of vertices in $V$ and the target density $d$.
	\begin{algorithm}[caption={Erd\"os-Renyi model for generating random bipartite graphs}, label={erdos_renyi_alg_01}]
		procedure erdos_renyi($n, s, d$)
		$\qquad$$U=\{u_1, \ldots, u_n\}$, $V=\{v_1, \ldots, v_n\}$, $E=\phi$;
		$\qquad$for each $u\in U$, $v\in V$:
		$\qquad$$\qquad$if (getBernoulli(d)==1) add edge $uv$ to $E$;
		$\qquad$return $G=(U\sqcup V, E)$;
		end
	\end{algorithm}
	The function \textit{getBernoulli(d)} returns $1$ with probability $d$ and $0$ with probability $(1-d)$. 
	It is not difficult to see that in the resulting graph $\rho(G)\approx d$.
	Observe that in this model, the degrees of the vertices satisfy that 
	$deg(u) \approx d\cdot s$ for $u\in U$
	and $deg(v) \approx d\cdot n$ for $v\in V$.
	In other words, we get almost null variability in the distribution of the degrees in both sides.
	
	The second model is the \textit{dispersed-degree} model. 
	It is designed to make the distribution of the degrees in the side $U$ more variable while preserving the target density $d$.
	Basically, we define the degree of each vertex $u\in U$ as an integral uniform random number in the interval 
	$[d\cdot s - r, d\cdot s + r]$, where $r$ is a custom \textit{dispersion radius}.
	Note that $r\leq s\cdot \min(d,1-d)$, to keep the degrees in the valid range $[0,s]$.
	Since different densities induce different upper bounds for the dispersion radius, then we define the
	\textit{normalized dispersion radius} as the quotient of the dispersion radius by the upper bound.
	Thus the normalized dispersion radius is always in the range $[0,1]$.
	Therefore, since the interval is centered at $d\cdot s$, then $\rho(G)\approx d$.
	Once we have defined the degree $deg(u)$ of a vertex $u\in U$, its neighbors are a random subset of $V$ of size $deg(u)$.
	The following pseudocode shows how to generate this class of graphs. 
	It takes as input: the number of vertices in $U$, the number of vertices
	in $V$, the desired density $d$, and the custom radius of dispersion $0\leq r \leq s\cdot \min(d,1-d)$.
	\begin{algorithm}[caption={Dispersed-degree model for random bipartite graphs}, label={disp_deg_alg_01}]
		procedure dispersed_degree($n, s, d, r$)
		$\qquad$$U=\{u_1, \ldots, u_n\}$, $V=\{v_1, \ldots, v_n\}$, $E=\phi$;
		$\qquad$for each $u\in U$ define $deg(u):=rand(d\cdot s-r,d\cdot s+r)$;
		$\qquad$for each $u\in U$:
		$\qquad$$\qquad$take $deg(u)$ random elements of $V$ as neighbors of $u$;
		$\qquad$return $G=(U\sqcup V, E)$;
		end
	\end{algorithm}
	
	For the structure of the random weights, we assume that the weights will be in the integer range $\{1,\ldots,100000\}$.
	We have three models for assigning random weights to the edges. 
	The first model is the \textit{uniform-weights model} that assigns to every edge an uniform random weight in 
	the range $\{1, \ldots, 100000\}$.
	The second model is the \textit{uniform-low-high-weights model}. This model randomly partitions the set of edges in two
	parts, the low-weights part and the high-weights part, according to a parameter $p\in[0,1]$. 
	Every edge has probability $p$ of being in the low-weights part, under a Bernoulli distribution,
	whose size is $\approx p\cdot |E|$ and the weights of these edges are chosen random uniform in the range $\{1,\ldots,1000\}$. The size of the high-weights part is $\approx (1-p)\cdot |E|$ and the weights
	of these edges are chosen random uniform in the range $\{1001,\ldots,100000\}$.
	The third model is the \textit{low-or-high-weights model} and is similar to the uniform-low-high-weights model,
	the difference is that the weights of the low-weights part are fixed at $1$ and the weights of the
	high-weights part are fixed at $100000$.
	
	From the time complexities of the algorithms we can see that if $|V|$ is asymptotically smaller than $|U|$, then the FlowAssign algorithm
	is expected to be on advantage. Furthermore, the random graphs that we will generate can be unbalanced, in this case the
	objective is to find a minimum-weight matching that covers all the vertices in the smaller side $V$.
	The Hungarian and the FlowAssign algorithms are designed to directly solve the assignment problem on unbalanced graphs, while
	the auction algorithm is not. The auction algorithm must address that problem by working on an induced balanced graph that has the double 
	of vertices and edges, therefore the first two algorithms are expected to be on advantage on unbalanced instances.
	
	\subsection{The experimental results}
	
	To generate the random instances to be solved by the algorithms, we defined a list of values for each parameter 
	and solved instances for all the possibilities of parameter configurations. The list of values are:
	\begin{enumerate}
		\itemsep-0.2em
		\item Edges distribution = \{erdos-renyi, dispersed-degree\},
		\item Costs distribution = \{uniform, uniform-low-high, low-or-high\},
		\item $n = \{1000, 2000, 4000, 8000\}$,
		\item $s = \{\log n, \sqrt{n}, n\}$,
		\item $Density = \{0.1, 0.5, 1.0\}$,
		\item Normalized dispersion radius $= \{0.1, 0.5, 1.0\}$ (only for dispersed degree),
		\item Low costs portion $= \{0.1, 0.5, 0.9\}$ (only for uniform-low-high and low-or-high).
	\end{enumerate}
	Note that we include families of asymptotically smaller values of $s$ respect to $n$ to explore the advantage of the FlowAssign algorithm's time complexity.
	Each instance is defined by picking one of each parameter.
	Since the instances are generated randomly and the computer can have different workload when solving an instance, we decided
	to take the solving time of a fixed set of parameters as the average solving time of solving ten different
	random instances generated under the same parameters.
	Also, in the following experimental results, given a subset of parameters at fixed values, its solving time is the average time of all
	the instances that share the same values in such parameters.
	The machine used to solve the instances is Windows 8.1 (64-bit), Intel i5-4670 at 3.4 GHz (4 CPUs) and 16 GB RAM.
	
	In Figure \ref{sensitivity_r}, we explore the sensitivity of the algorithms respect to the normalized dispersion radius parameter for the dispersed-degree random graph model.
	We can observe that the auction and FlowAssign algorithms perform worst at high normalized dispersion radius, while the Hungarian algorithm seems to benefit a little 
	at high values of the parameter.

	In Figure \ref{sensitivity_p_loh}, we explore the sensitivity of the low-costs edges portion for the low-or-high-weights costs model.
	Its interesting that this parameter seems completely irrelevant for the Hungarian algorithm, while the FlowAssign algorithm seems to perform significantly
	better at high values of the parameter. For the auction algorithm, the best value seems to lie somewhere in the middle.
	This behavior may be influenced from the fact that the probability that the low-cost edges contain a (one-side) perfect matching increasing quickly with the number
	of edges, due to the way the edges are randomly generated.

	In Figure \ref{sensitivity_p_uloh}, we explore the sensitivity of the low-costs edges portion for the uniform-low-high-weights costs model.
	This time the Hungarian algorithm is sensitive to the parameter and performs worst at high values of the parameter, as well as the auction algorithm.
	The FlowAssign algorithm behaves similar than in the low-or-high-weights costs model.

	In Figure \ref{sensitivity_edges}, we explore the sensitivity of the algorithms to the random graphs models.
	In this case the auction algorithm performs better in the dispersed-degree random graphs model, while the FlowAssign
	and Hungarian algorithm performs better in the Erd\"os-Renyi model.
	This behavior can be related to the normalized dispersion radius.

	In Figure \ref{sensitivity_costs}, we explore the sensitivity of the algorithms to the random costs models.
	For this parameter the behavior is interesting. For the auction algorithm, we can observe a significantly worst
	performance in the low-or-high costs model but only for very unbalanced graphs, while for balanced graphs
	the algorithm performs its worst with the uniform-low-high costs model.
	Interestingly, the FlowAssign algorithm benefits a lot from the low-or-high costs model, and performs equal for the
	other two models.
	And the Hungarian algorithm performs a little better in the uniform costs model than in the other models.
	Remember that the auction algorithm has to solve the problem in an auxiliary graph with double number of vertices and edges
	when the graph is unbalanced.
	
	Note the huge solving time difference between the algorithms in all the graphs presented so far.
	The performance of the auction algorithm is outstanding in every scenario, the performance of the FlowAssign algorithm is acceptable and the performance of the Hungarian algorithm is extremely poor.

	In Figure \ref{sensitivity_unbalanced}, we explore the sensitivity of the algorithms to the unbalanceness gap, that is,
	to the asymptotic of $s=|V|$ respect to $n=|U|$.
	In this case we can observe that for $s=\log n$, the Hungarian algorithm performs better in average, followed by
	the Auction algorithm and then the Flow Assign algorithm, all of them with a significant difference.
	For $s=\sqrt{n}$, the FlowAssign and Hungarian algorithm perform similar between them and a lot better than the Auction algorithm.
	Remember that the auction algorithm is in big disadvantage in the unbalanced case in more than one way.
	For the balanced case $s=n$, the Auction algorithm is in first place, followed by FlowAssign by a significant factor and 
	then Hungarian by a huge factor.

	\section{Conclusions}
	
	We have shown that two apparently different algorithms are rather the same under a few optimizations in allocation space.
	In other words, if we change a little bit the G\&K algorithm in order to get off of redundant objects, we end up exactly with
	the auction algorithm.
	Therefore any heuristic developed for one of the algorithms can be implemented in the other without difficulty.
	Also the auction algorithm automatically induces a parallel implementation \cite{paral_auction1,paral_auction2} on the G\&K algorithm.
	
	From the performance analysis, we conclude that the Auction algorithm performs impressively better against other algorithms, even under conditions that put it in disadvantage.
	The only cases where it performed worst than its competitors is where the time required to solve the instances is very low.
	But in general the auction algorithm outperformed its competitors by a big factor, even to the FlowAssign algorithm which has much better
	time complexity and is designed to exploit the unbalanceness gap of the bipartite graphs.

	\section*{Acknowledgments}
	This research was partially supported by SNI and CONACyT.

	%
	%
	
	\begin{figure}[h]
		\begin{center}
			\begin{tabular}{@{\hspace{-1cm}} l @{\hspace{-1.1cm}}  @{\hspace{-1.2cm}} l @{\hspace{-1cm}}}
				\includegraphics[scale=0.45]{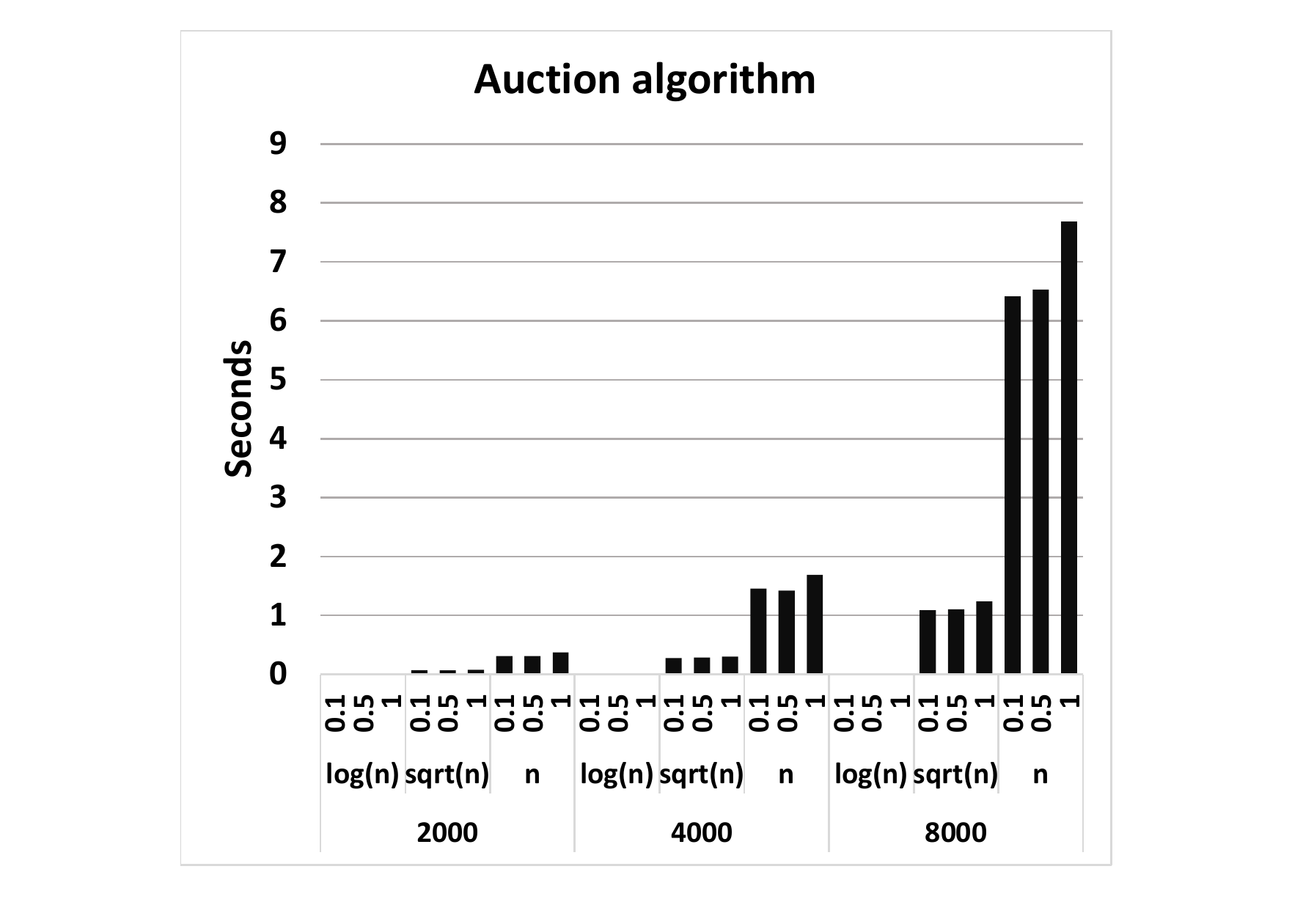} & \includegraphics[scale=0.45]{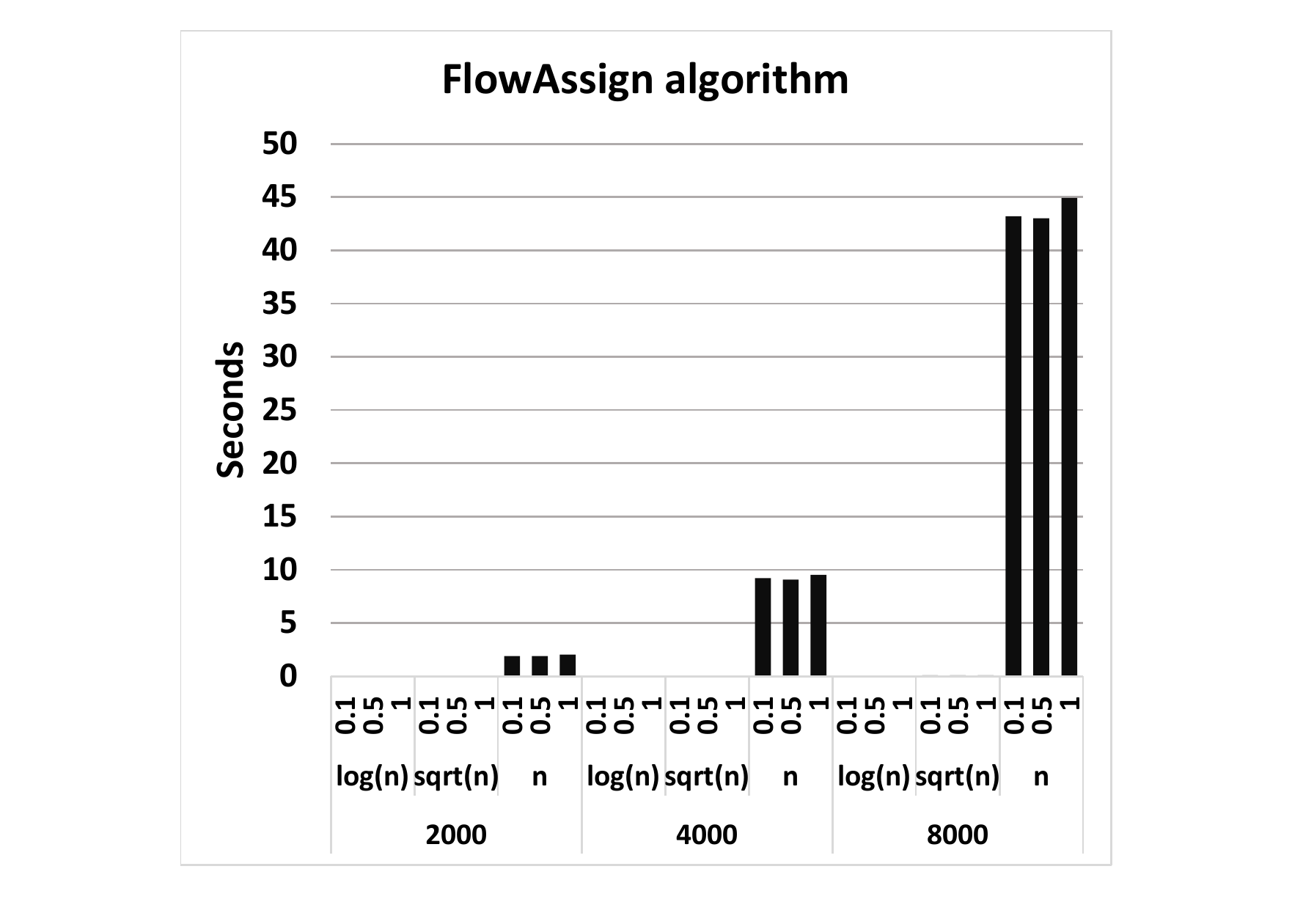} \\[-0.6cm]
				\multicolumn{2}{c}{\includegraphics[scale=0.45]{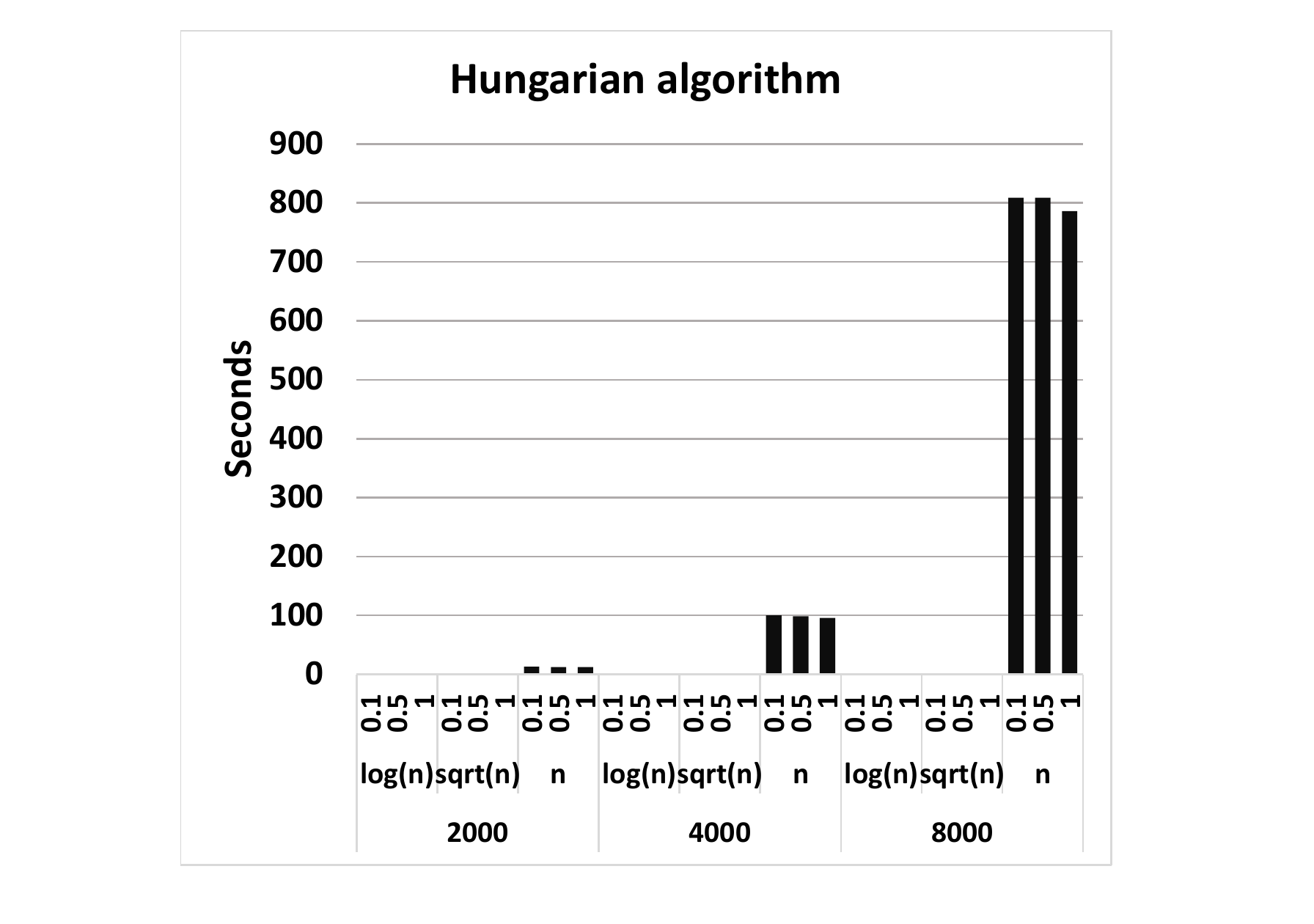}} \\
			\end{tabular}
		\end{center}
		\caption{Sensitivity analysis of the normalized dispersion radius parameter.}
		\label{sensitivity_r}
	\end{figure}
	
	\begin{figure}[h]
		\begin{center}
			\begin{tabular}{@{\hspace{-1cm}} l @{\hspace{-1.1cm}}  @{\hspace{-1.2cm}} l @{\hspace{-1cm}}}
				\includegraphics[scale=0.45]{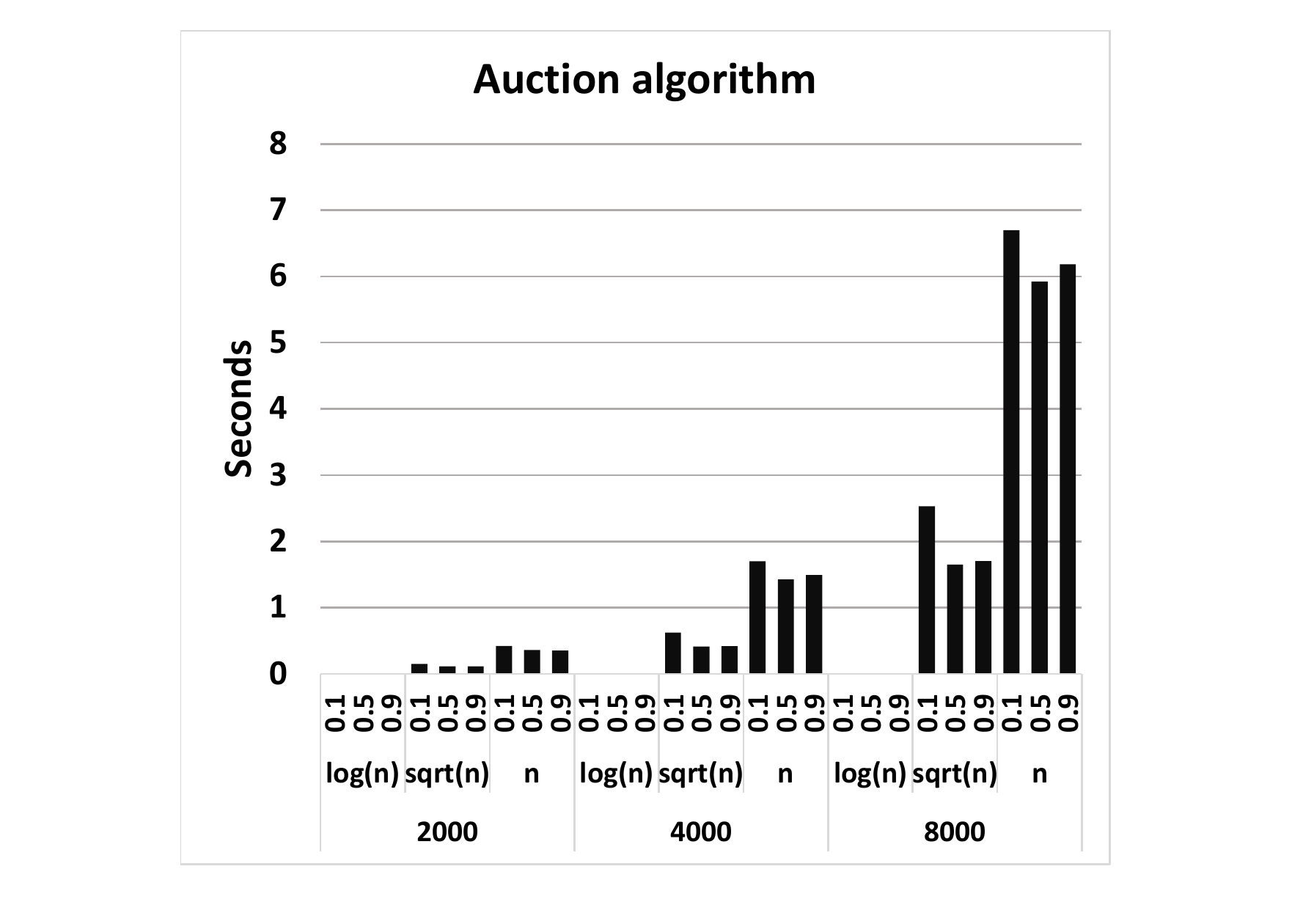} & \includegraphics[scale=0.45]{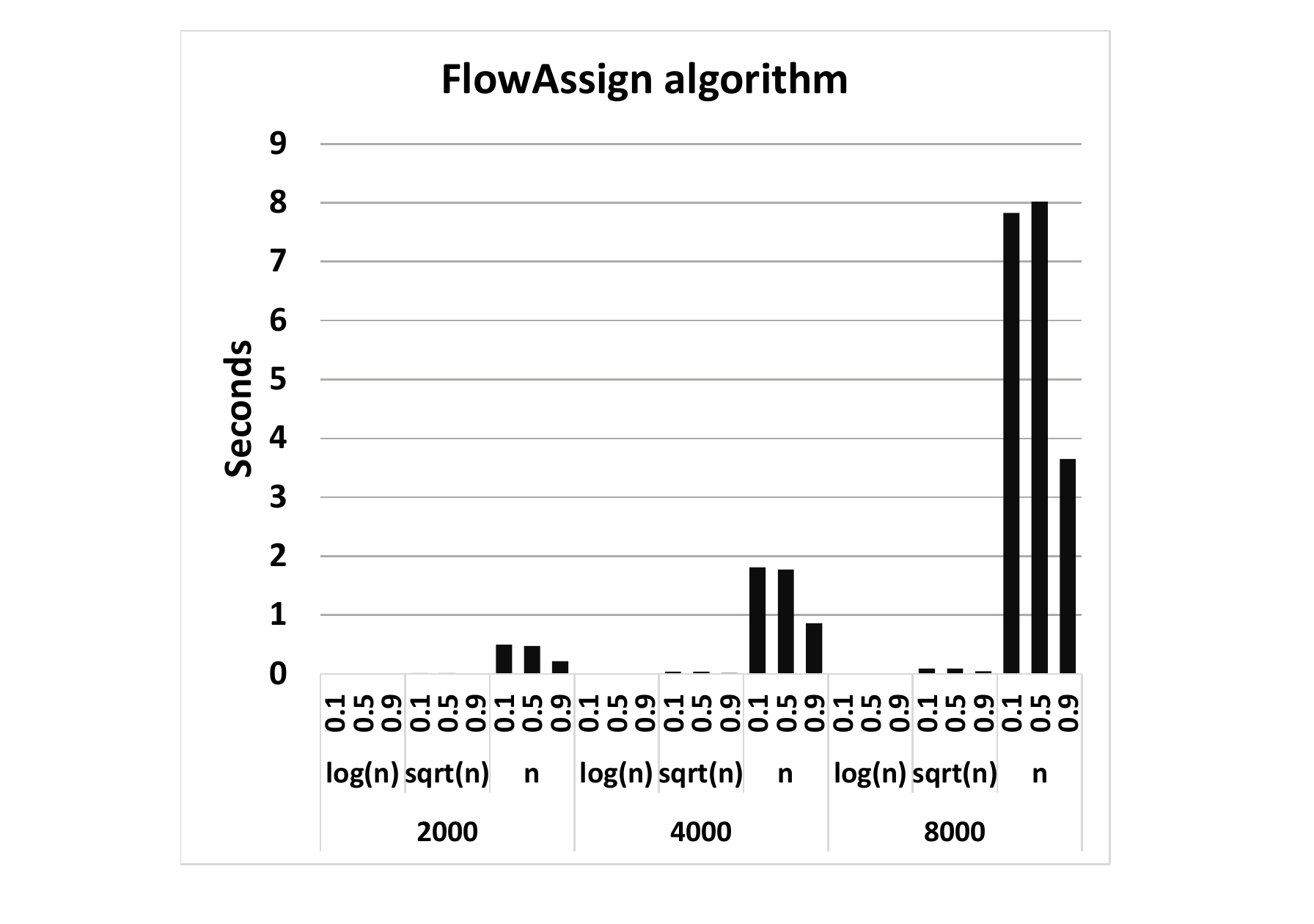} \\[-0.6cm]
				\multicolumn{2}{c}{\includegraphics[scale=0.45]{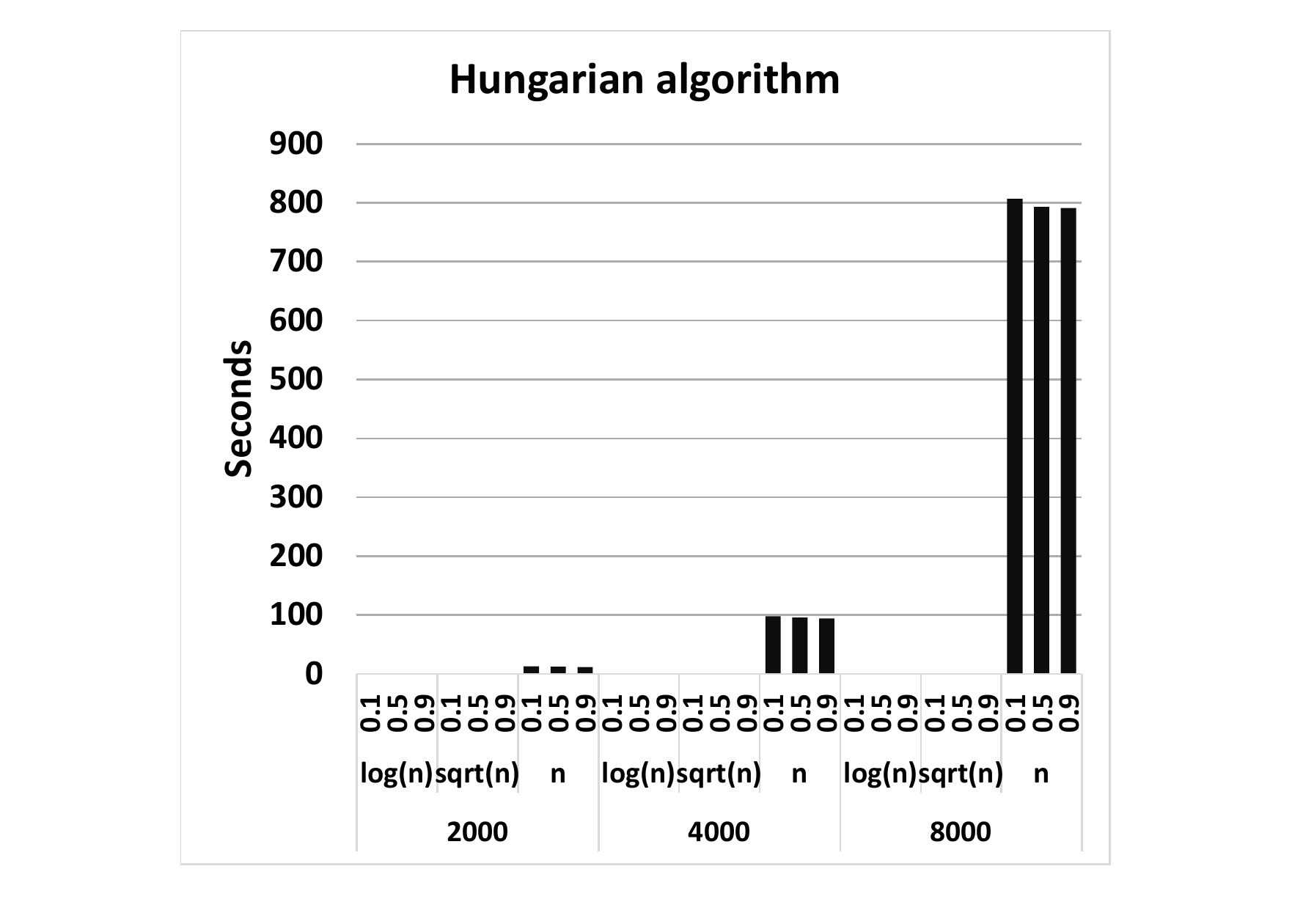}} \\
			\end{tabular}
		\end{center}
		\caption{Sensitivity analysis of the portion of low cost edges for the low-or-high-weights costs model.}
		\label{sensitivity_p_loh}
	\end{figure}
	
	\begin{figure}[h]
		\begin{center}
			\begin{tabular}{@{\hspace{-1cm}} l @{\hspace{-1.1cm}}  @{\hspace{-1.2cm}} l @{\hspace{-1cm}}}
				\includegraphics[scale=0.45]{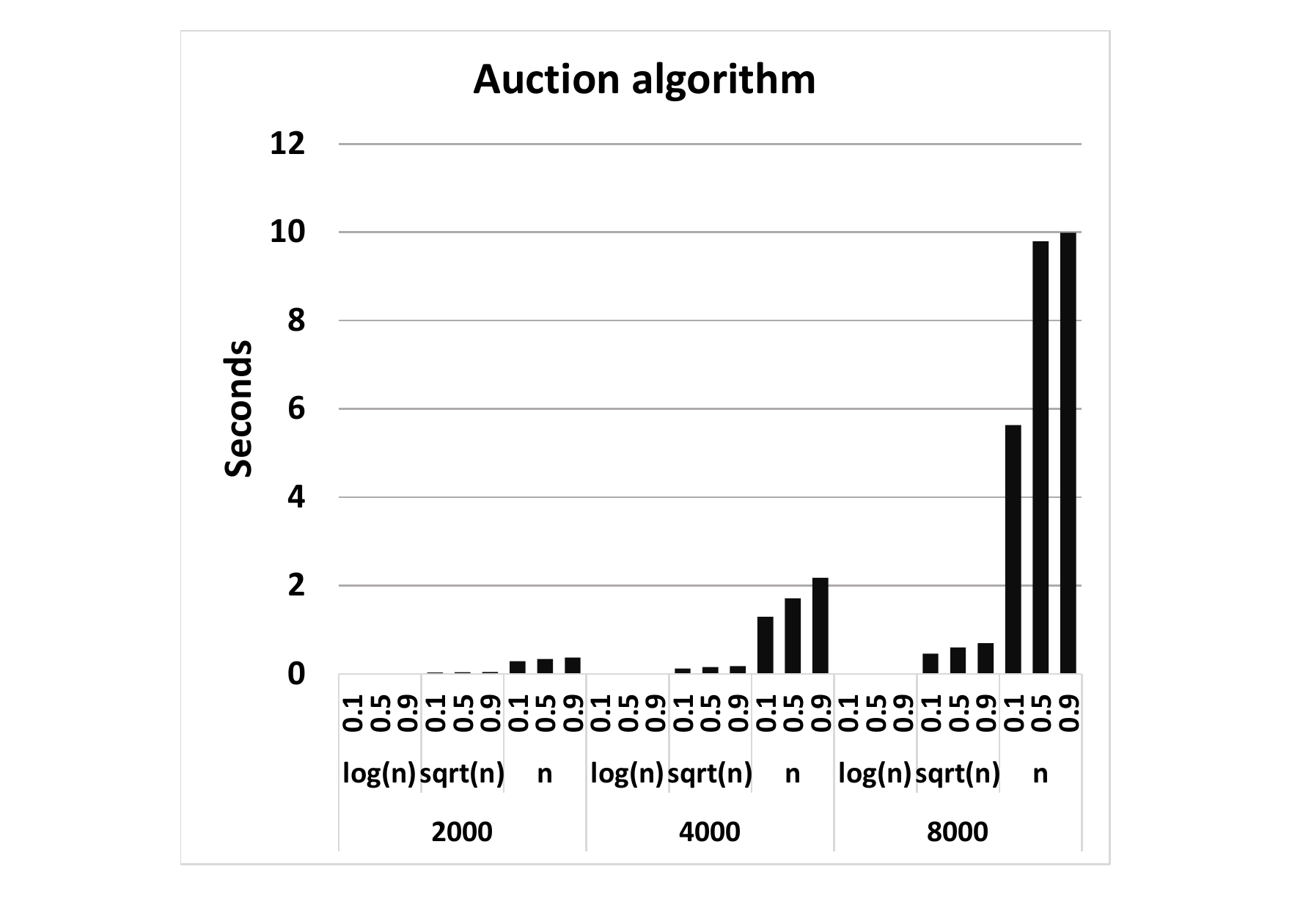} & \includegraphics[scale=0.45]{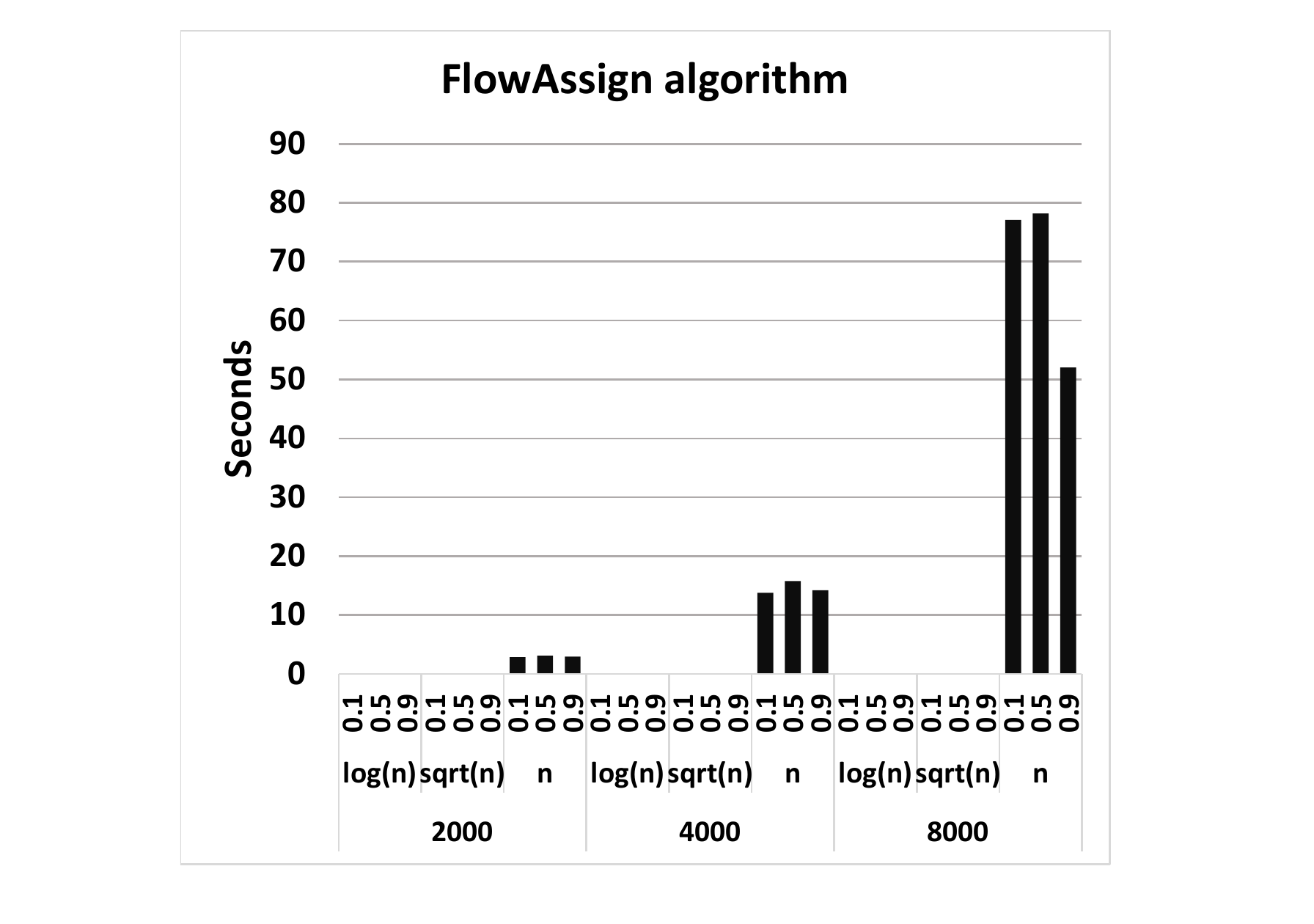} \\[-0.6cm]
				\multicolumn{2}{c}{\includegraphics[scale=0.45]{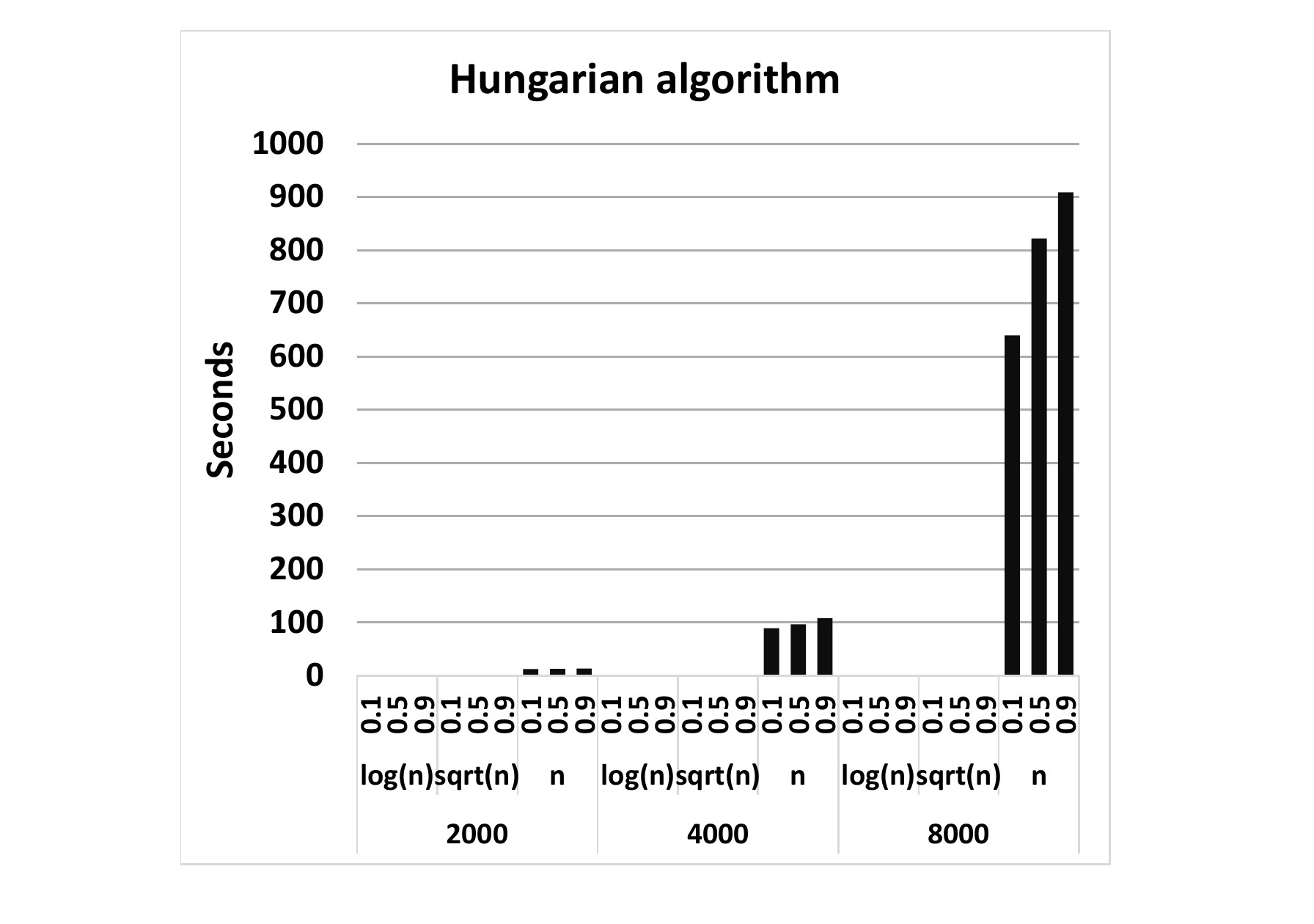}} \\
			\end{tabular}
		\end{center}
		\caption{Sensitivity analysis of the portion of low cost edges for the uniform-low-high-weights costs model.}
		\label{sensitivity_p_uloh}
	\end{figure}
	
	\begin{figure}[h]
		\begin{center}
			\begin{tabular}{@{\hspace{-1cm}} l @{\hspace{-1.1cm}}  @{\hspace{-1.2cm}} l @{\hspace{-1cm}}}
				\includegraphics[scale=0.45]{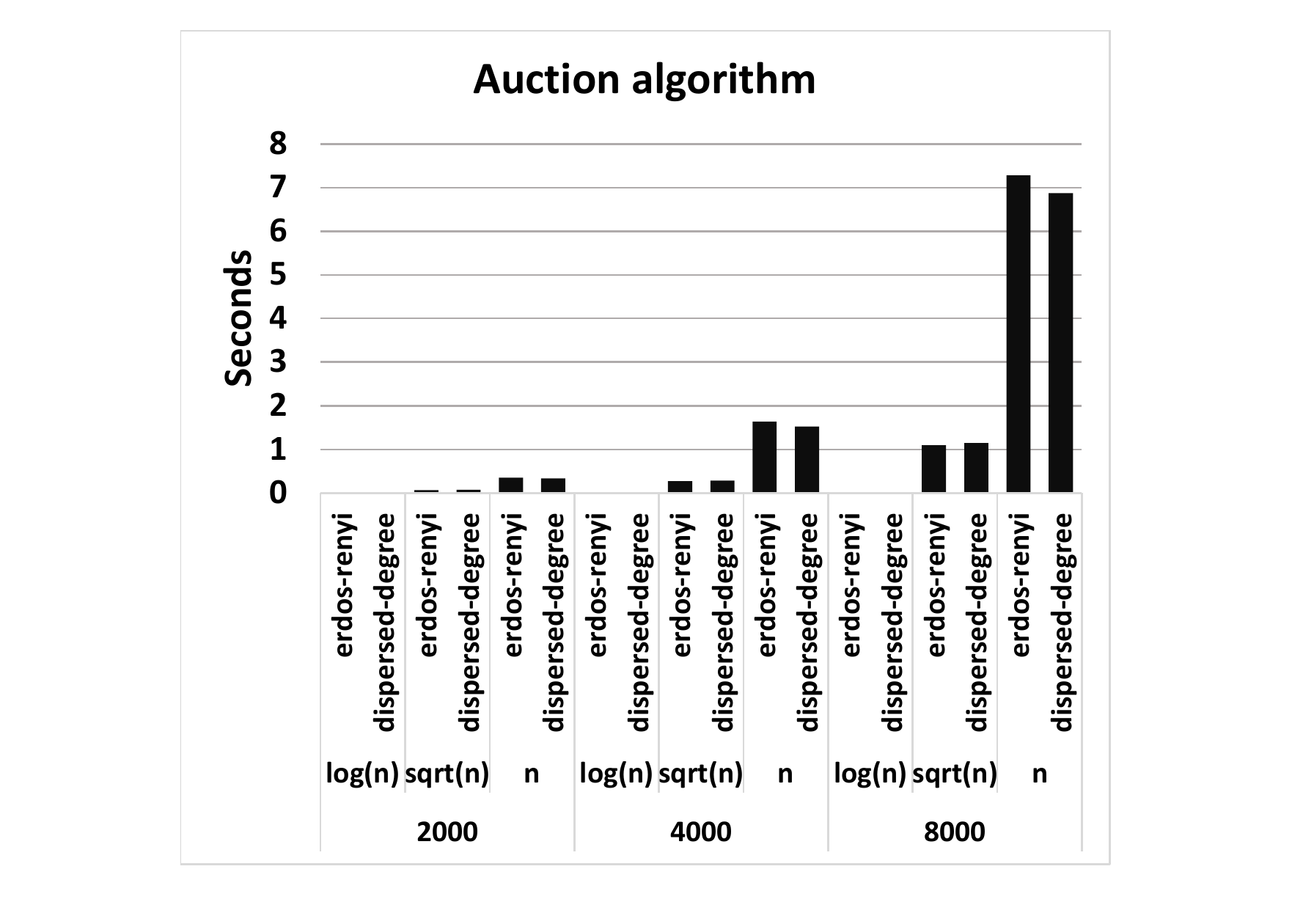} & \includegraphics[scale=0.45]{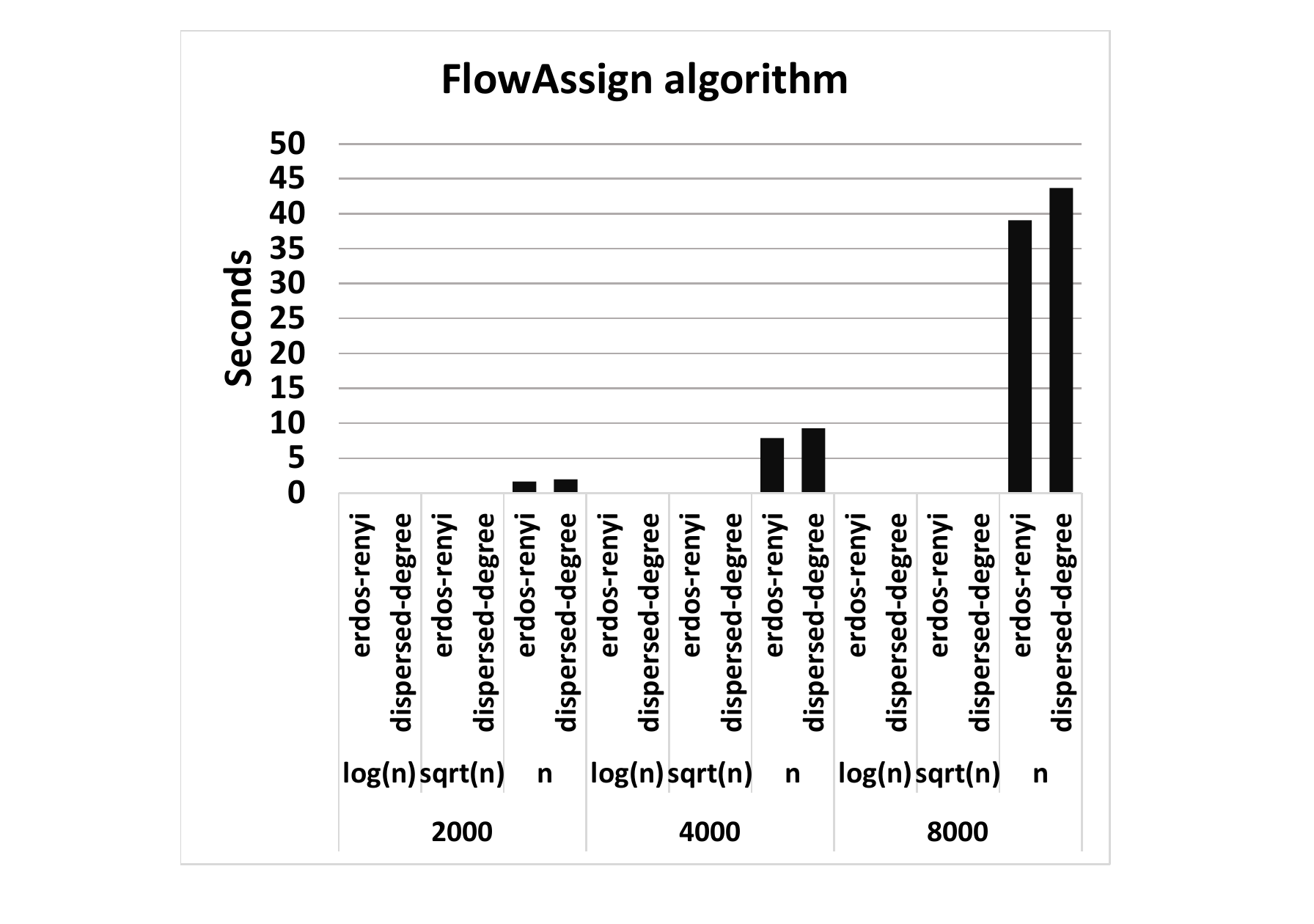} \\[-0.6cm]
				\multicolumn{2}{c}{\includegraphics[scale=0.45]{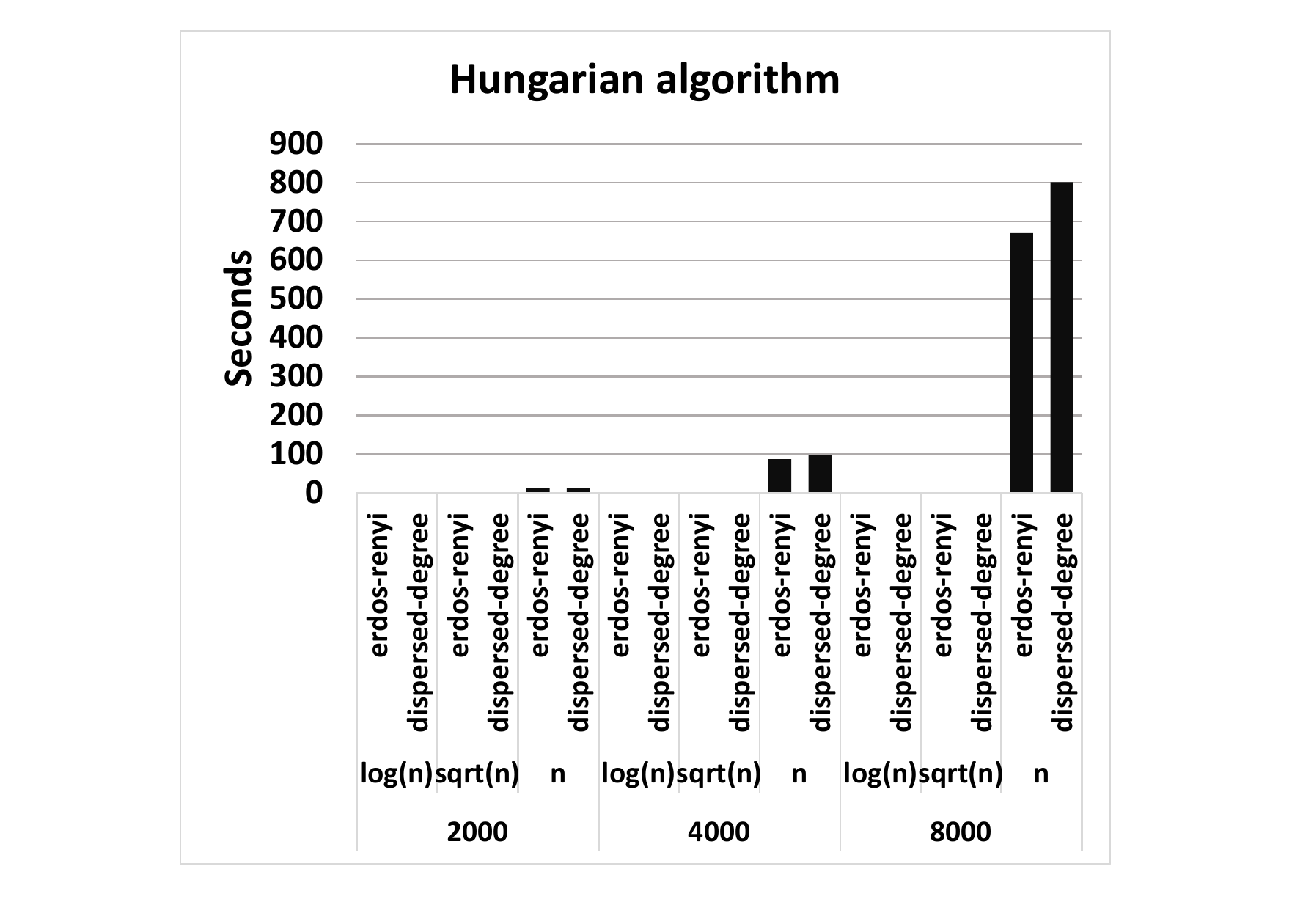}} \\
			\end{tabular}
		\end{center}
		\caption{Sensitivity analysis of the random graphs models.}
		\label{sensitivity_edges}
	\end{figure}
	
	\begin{figure}[h]
		\begin{center}
			\begin{tabular}{@{\hspace{-1cm}} l @{\hspace{-1.1cm}}  @{\hspace{-1.2cm}} l @{\hspace{-1cm}}}
				\includegraphics[scale=0.45]{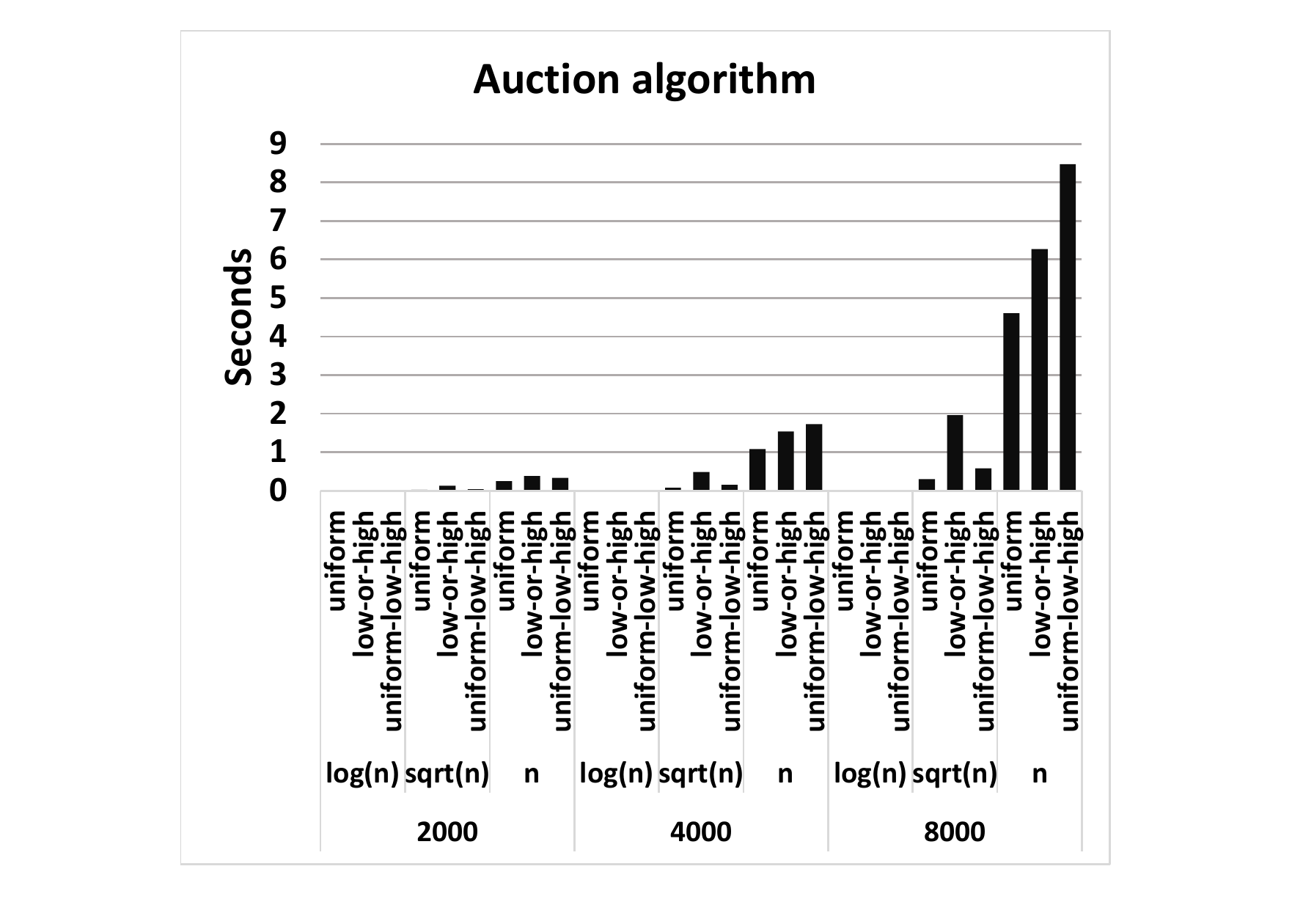} & \includegraphics[scale=0.45]{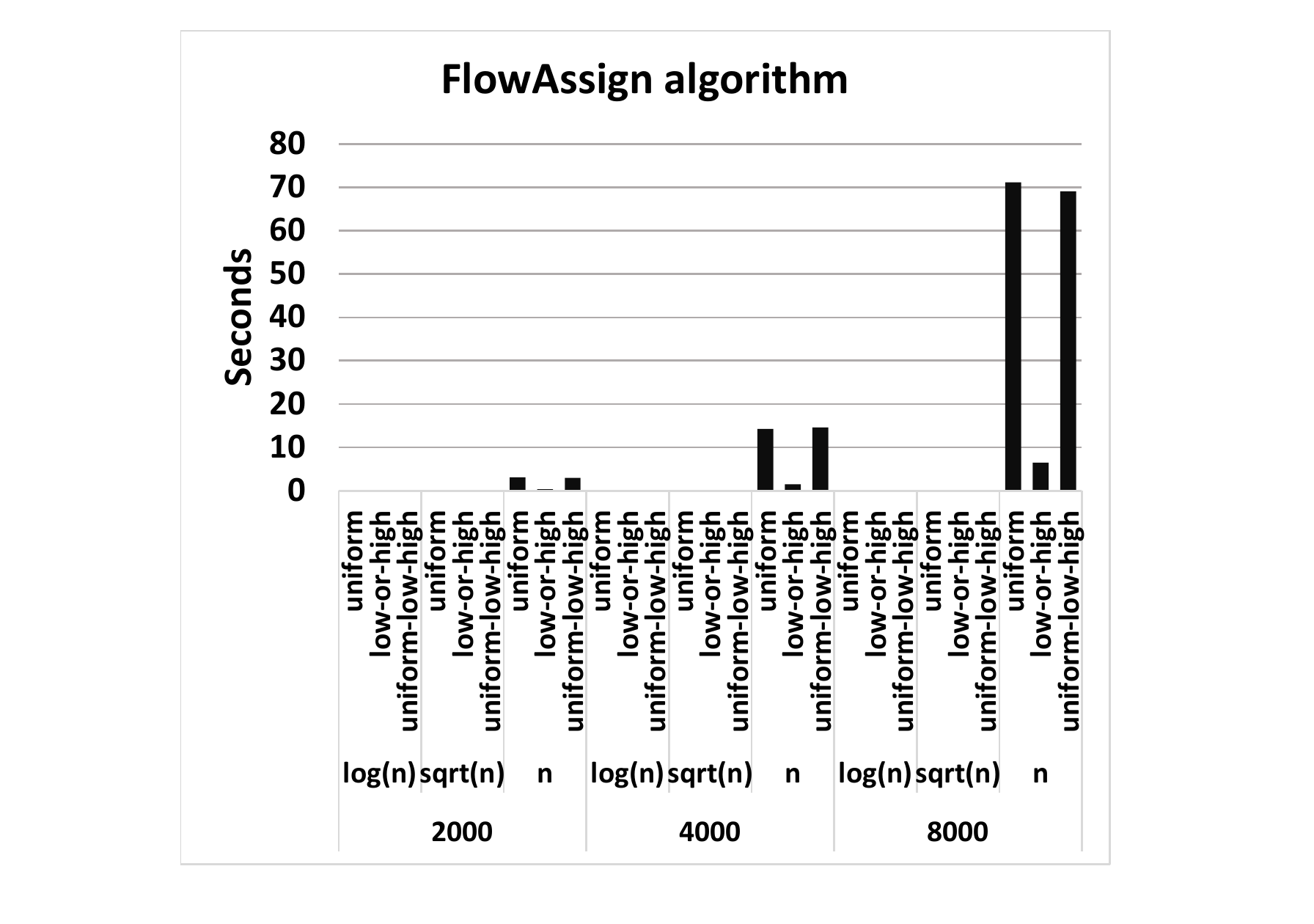} \\[-0.6cm]
				\multicolumn{2}{c}{\includegraphics[scale=0.45]{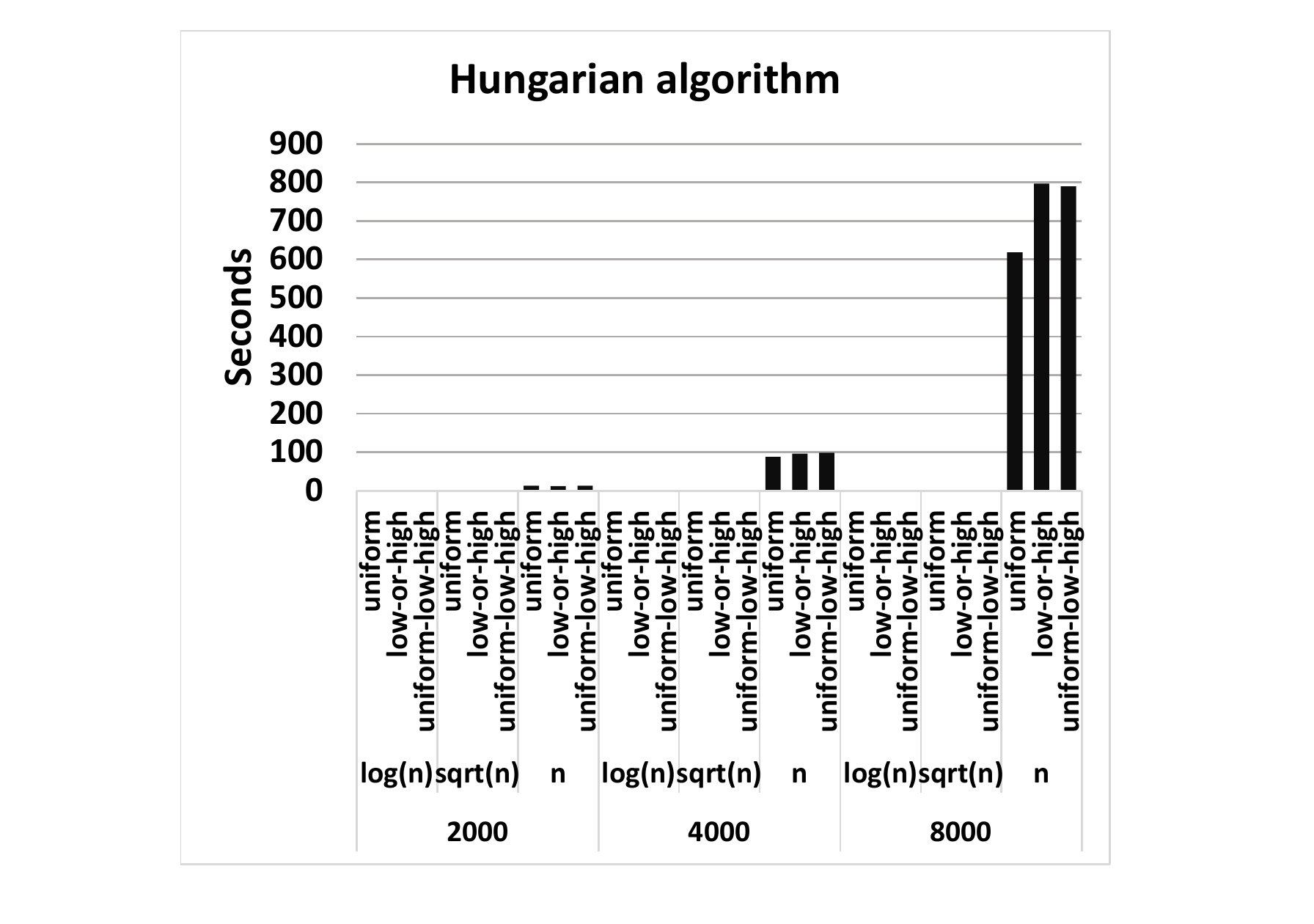}} \\
			\end{tabular}
		\end{center}
		\caption{Sensitivity analysis of the random costs models.}
		\label{sensitivity_costs}
	\end{figure}
	
	\begin{figure}[h]
		\begin{center}
			\begin{tabular}{@{\hspace{-1cm}} l @{\hspace{-1.1cm}}  @{\hspace{-1.2cm}} l @{\hspace{-1cm}}}
				\includegraphics[scale=0.45]{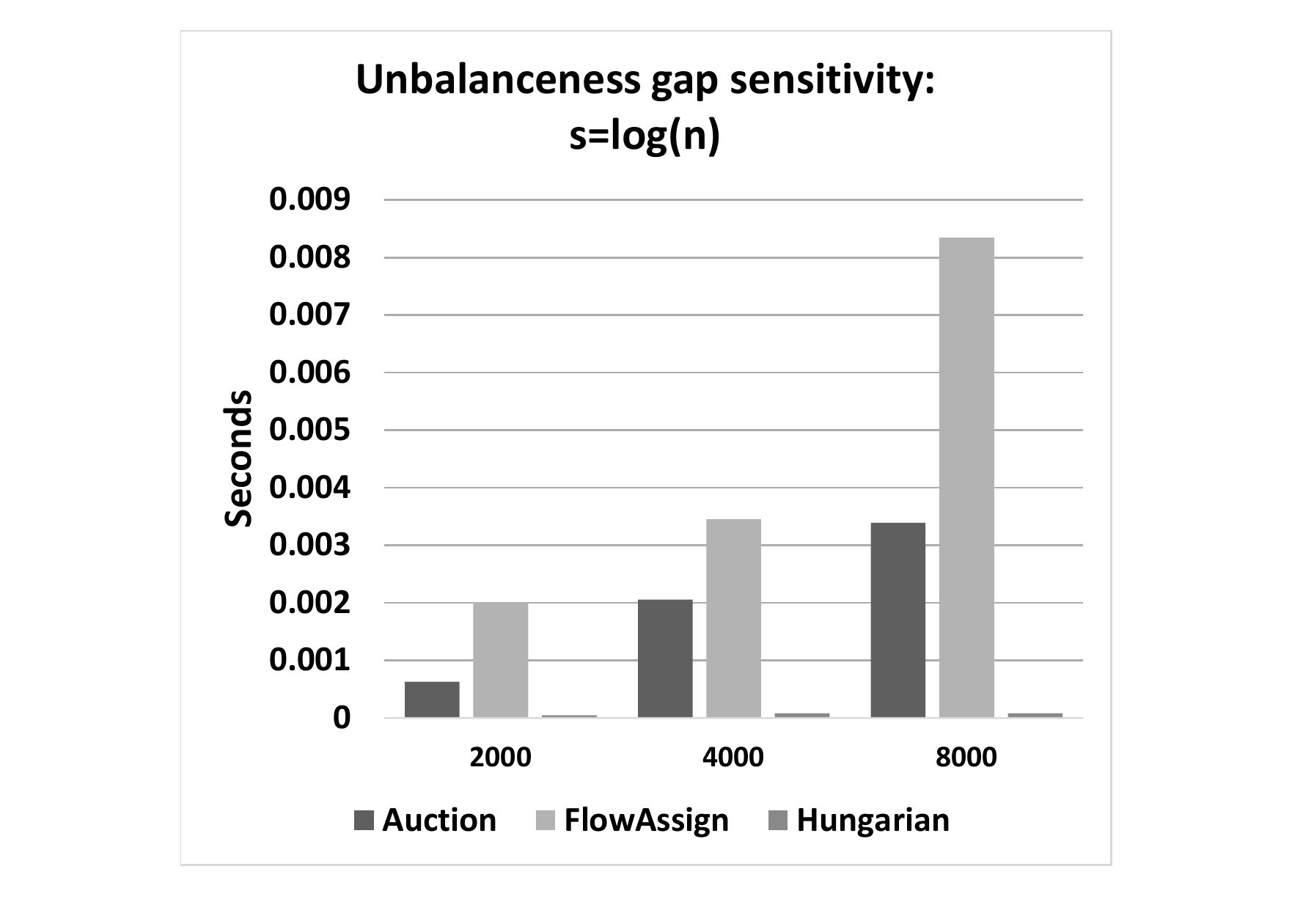} & \includegraphics[scale=0.45]{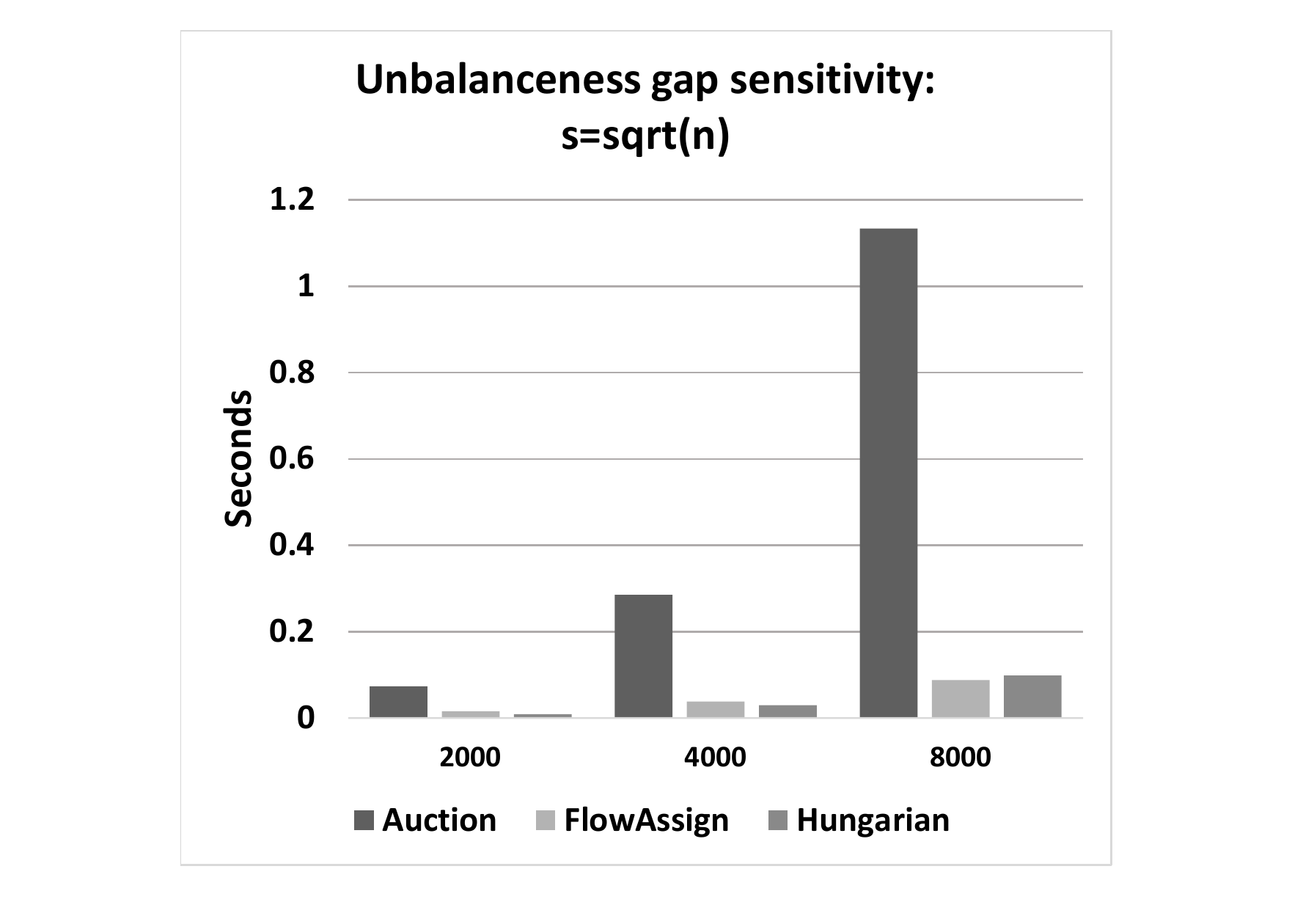} \\[-0.6cm]
				\multicolumn{2}{c}{\includegraphics[scale=0.45]{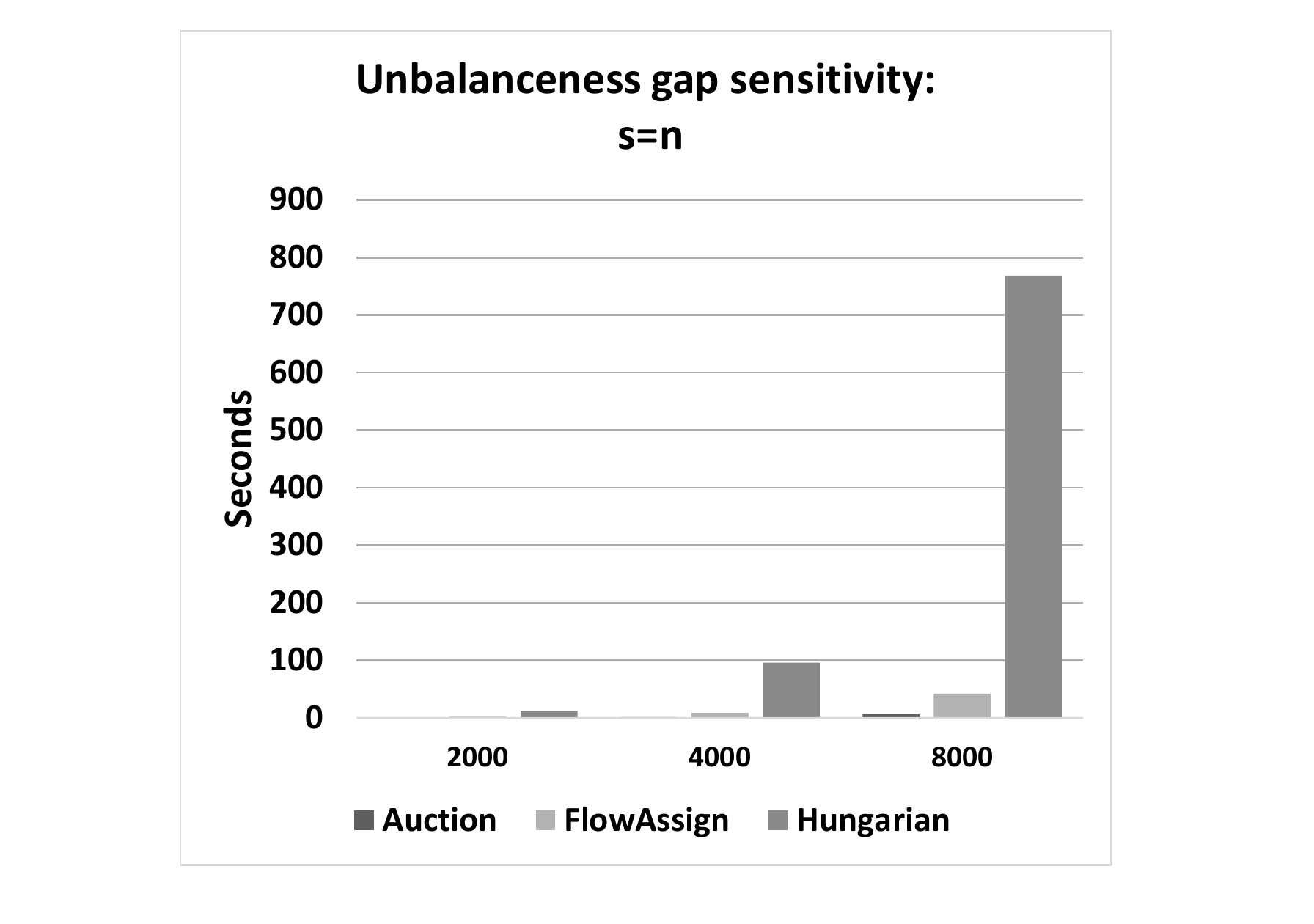}} \\
			\end{tabular}
		\end{center}
		\caption{Sensitivity analysis of the asimptoticity of $s=|V|$ respect to $n=|U|$.}
		\label{sensitivity_unbalanced}
	\end{figure}

\end{document}